\documentclass[a4paper]{amsart}

\usepackage[utf8]{inputenc}
\usepackage[T1]{fontenc}
\usepackage[english]{babel}

\usepackage[norelsize]{algorithm2e}

\usepackage{booktabs}
\usepackage{url}
\usepackage{longtable}

\newtheorem{theorem}{Theorem}[section]
\theoremstyle{definition}
\newtheorem{definition}[theorem]{Definition}
\theoremstyle{remark}
\newtheorem{remark}[theorem]{Remark}

\newtheorem{lemma}[theorem]{Lemma}
\newtheorem{corollary}[theorem]{Corollary}
\newtheorem{proposition}[theorem]{Proposition}

\newcommand{\MdF}{\mathbb{F}}

\newcommand{\MdN}{\mathbb{N}}

\newcommand{\MdZ}{\mathbb{Z}}
\newcommand{\Z}{\mathbb{Z}}
\newcommand{\MdQ}{\mathbb{Q}}
\newcommand{\Q}{\mathbb{Q}}

\newcommand{\MdR}{\mathbb{R}}

\newcommand{\MdP}{\mathbb{P}}

\newcommand{\mass}{\text{mass}}
\newcommand{\Gen}{\text{Genus}}

\newcommand{\Wat}{\text{Wat}}

\newcommand{\sym}{\text{sym}}

\newcommand{\Watp}{\text{Wat}_p}

\newcommand{\Rescale}{\text{rescale}}
\newcommand{\rank}{\text{rank}}
\newcommand{\maxprime}{\text{maxprime}}

\newcommand{\Jordan}{\text{Jordan}}

\newcommand{\iv}{^{-1}}

\newcommand{\rarr}{\rightarrow}

\newcommand{\disc}{\text{disc}}
\newcommand{\erz}[1]{\ensuremath{\langle #1\rangle}}

\newcommand{\Aut}{\text{Aut}}
\newcommand{\MAGMA}{\textsc{Magma} }
\newcommand{\Gram}{\text{Gram}}

\newcommand{\std}{\text{std}}

\newcommand{\Quot}{\text{Quot}}

\begin{document}

\markboth{David Lorch}
{Single-Class Genera of Positive Integral Lattices}

\title{Single-Class Genera of Positive Integral Lattices}

\author{DAVID LORCH}

\address{Lehrstuhl D f\"ur Mathematik, RWTH Aachen University}
\email{david.lorch@math.rwth-aachen.de}

\author{MARKUS KIRSCHMER}

\address{Lehrstuhl D f\"ur Mathematik, RWTH Aachen University}
\email{markus.kirschmer@math.rwth-aachen.de}

\begin{abstract}

We give an enumeration of all positive definite primitive $\MdZ$-lattices in dimension $\geq 3$ whose genus consists of a single isometry class. This is achieved by using bounds obtained from the \textsc{Smith-Minkowski-Siegel} mass formula to computationally construct the square-free determinant lattices with this property, and then repeatedly calculating pre-images under a mapping first introduced by \textsc{G.\,L.\,Watson}.\\ We hereby complete the classification of single-class genera in dimensions $4$ and $5$ and correct some mistakes in Watson's classifications in other dimensions. A list of all single-class primitive $\MdZ$-lattices has been compiled and incorporated into the Catalogue of Lattices.

\end{abstract}

\maketitle


\section{Introduction}

There has been extensive research on single-class lattices. \textsc{G.\,L.\,Watson} proved in \cite{watsonsingleclass} that single-class $\MdZ$-lattices cannot exist in dimension $n>10$ and, in a tremendous effort, tried to compile complete lists of single-class lattices in dimensions $3$~--~$10$ using only elementary number theory (\cite{ternary1, watson5, quaternary, ternary2, watson910, watson8, watson7, watson6}). While he succeeded in classifying most of these lattices, he found the case of dimension $4$ and $5$ to be exceedingly difficult and classified only a subset of the single-class lattices. 
In this work, all of the lattices missing from his classification have been found for what we believe is the first time. It turns out that, aside from some omissions in dimensions $3$ to $6$, \textsc{Watson}'s results are largely correct.


Partial improvements to his results have already been published, notably for the three-dimensional case in an article by \textsc{Jagy} et al., \cite{jagy}, whose results agree with our computation. Another such improvement concerns the subset of single-class $\MdZ$-lattices in dimensions $3$~--~$10$ which correspond to a maximal primitive quadratic form -- this has been enumerated by \textsc{Hanke} in \cite{hanke}. Again, these results agree with ours.

In dimension~$2$, single-class $\MdZ$-lattices have been classified by capitalizing on a connection to class groups of imaginary quadratic number fields, cf.~Voight in~\cite{voight}. This classification is proven complete if the Generalized Riemann Hypothesis holds.

\subsection{Statement of results}
We give an algorithm for finding, up to equivalence, all primitive positive definite $\MdZ$-lattices with class-number $1$ in dimension $\geq 3$. Computation on genera of lattices is performed by means of 
a \emph{genus symbol} 
developed by \textsc{Conway} and \textsc{Sloane} (\cite[Chapter~15]{SPLAG}). A mass formula, stated by the same authors in \cite{massformula} but dating back to contributions by  \textsc{Smith}
, \textsc{Minkowski} 
and \textsc{Siegel} 
provides effective bounds to the number of local invariants that need to be taken into account, thus making the enumeration computationally feasible.

Our main result, the complete list of single-class primitive $\MdZ$-genera, and representative lattices of each of these genera, has been incorporated into the Catalogue of Lattices \cite{latdb}. Tables containing only the genus symbols, and only in dimension $\geq 4$, can be found in the appendix.

An additional contribution of this work is a number of essential algorithms for computation on genera of $\MdZ$-lattices. All of these algorithms have been implemented in {\textsc{Magma}} \cite{magma} and are available on request.

%


\section{Preliminaries}

\subsection*{Genera of lattices}\label{subsec:preliminaries}

We denote by $\MdP$ the set of rational primes. In the following, let $R=\MdZ$ or $R=\MdZ_p$ for some $p\in \MdP$. By an $R$-{lattice} $L$ we mean a tuple $(L,\beta)$, where $L=\erz{b_1,\dots,b_n}$ is a free $R$-module of finite rank $n$ and $\beta:L\times L\rarr \Quot(R)$ is a symmetric bilinear form. Additionally, if $R=\MdZ$, we assume that $\beta$ is positive definite. $L$ is {integral} if $\beta(L,L)\subseteq R$. An integral $R$-lattice is called {even} if $\beta(L,L)\subseteq 2R$, and {odd} otherwise. $L$ is quadratic-form-maximal if either $L$ is even and the associated quadratic form is maximal integral, or if $L$ is odd and the quadratic form associated to $2L$ is maximal integral. By $\Aut(L)$ we mean the group of bijective isometries of $L$.

The {Gram matrix} of $L$ is $\Gram(L):=(\beta(b_i,b_j))_{1\leq i,j\leq n}$. By the determinant $\det(L)$ we mean $\det(\Gram(L))$, and the discriminant of $L$ is $\disc(L)=(-1)^s\det(L)$ where $s:=\left\lceil\frac{\rank(L)}{2}\right\rceil$.

Recall the definition of the genus of a lattice: two $\MdZ$-lattices $L$, $L'$ are said to be in the same genus if their completions $\MdZ_p\otimes_\MdZ L$ and $\MdZ_p\otimes_\MdZ L'$ are isometric for every $p\in\MdP\cup\{-1\}$, with the convention that $\MdZ_{-1}:=\MdR$. It is well known that genera of lattices are finite unions of isometry classes. We denote the genus of $L$ by $\Gen(L)$ and call the number of isometry classes contained in $\Gen(L)$ the class-number of~$L$.

If $L$ is an $R$-lattice with Gram matrix $G=(g_{ij})_{1\leq i,j\leq n}\in R^{n\times n}$, we call $L$ {primitive} if $\gcd(\{g_{ij}:\ 1\leq i,j\leq n\})=1$. By $\Rescale(L)$ we mean the unique scaled lattice ${^\alpha}L:=(L,\alpha\beta)$ with the property that $\Gram({^\alpha}L)$ is integral and primitive.

By $L^\#=\{ax:\ x\in L, a\in \Quot(R), \beta(ax, L)\subseteq R\}$ we mean the dual of~$L$. Whenever, for some integer $a\in\MdN$, $aL^\# = L$, we call $L$ $a$-modular. If a $\MdZ$-lattice $L$ is integral, or equivalently if $L\subseteq L^\#$, we will call $L$ {$p$-adically square-free} whenever $\exp(L^\#/L) \not\equiv 0\mod p^2$. Here, $\exp(L^\#/L)$ denotes the exponent of the finite abelian group $L^\#/L$. If $L$ is $p$-adically square-free for all $p\in\MdP$, we will call $L$ square-free.

\subsection*{Jordan splitting}We remind the reader of the {Jordan splitting} of $\MdZ_p$-lattices: let $L$ be a $\MdZ$-lattice and $p\in\MdP$. Then  $\MdZ_p\otimes_\MdZ L$ admits an orthogonal splitting $\MdZ_p\otimes_\MdZ L=\bot_{i\in\MdZ}L_i$ with $L_i$ $p^i$-modular (but possibly zero-dimensional). When $L$ is integral, $i$ ranges over $\MdN_0$ only. 

For $2\not=p\in\MdP$, the $p$-adic Jordan decompositions of $L$ are unique up to isometry of the components. Much of the complication in computations with genera of lattices arises from the fact that a $\MdZ_2$-lattice can have many essentially different Jordan decompositions.

The dimensions of the $L_i$ are, however, unique even in the case $p=2$. If $a=\min\{i\in\MdZ:\ \rank(L_i)>0\}$ and $b=\max\{i\in\MdZ:\ \rank(L_i)>0\}$, then we define the Jordan decomposition's {length} to be $\text{len}_p(L):=b-a+1$. By $\Jordan_p(L)=L_a\bot \cdots\bot L_b$ we mean that the right side is {a} $p$-adic Jordan decomposition of $\MdZ_p\otimes_\MdZ L$, with $L_i$ a $p^i$-modular lattice for each $a\leq i\leq b$.

If $L$ is $p$-adically square-free, then clearly all $L_i$ have dimension $0$ for $i\geq 2$. If $L$ is square-free and, in addition, for all $p\in \MdP$ we have $\Jordan_p(L)=L_0\bot L_1$ with $\rank(L_0)\geq \rank(L_1)$, we call $L$ {strongly primitive}.

\subsection*{The genus symbol}

There have been many descriptions of complete sets of real and $p$-adic invariants of genera of $\MdZ$-lattices $L$. We will adopt the notion of a {genus symbol} $\text{sym}(L)$ that has been put forth by \textsc{Conway} and \textsc{Sloane} in \cite[Chapter~15]{SPLAG}, because it appears to us to be the most concise in terms of understanding and computing the $2$-adic invariants. 

We assume basic familiarity with this symbol and only note that $\sym(L)$ is a formal product of lists of tuples for each prime $p$ dividing $2\det(L)$, with each tuple containing $i$, $\rank(L_i)$, $\det(L_i)\bmod{(\MdZ_p^*)^2}$ corresponding to a $p^i$-modular orthogonal summand $L_i$ of $\Jordan_p(L)$, and (for $p=2$) an invariant taking either a value based on $\text{trace}(L_i)\bmod{8}$ (called the {oddity} of $L_i$), or the value ``II'' if $L_i$ is an even $\MdZ_2$-lattice. These tuples are abbreviated as $(p^i)^{\varepsilon \rank(L_i)}_\text{oddity}$, with the subscript present only for $p=2$. For $p\not=2$, $\varepsilon\in\{-1,1\}\cong (\MdZ_p^*)/(\MdZ_p^*)^2$ denotes the square-class of $\det(L_i)$. For $p=2$, we write $\varepsilon=-1$ for $\det(L_i)\equiv\pm 3\bmod 8$, and $\varepsilon=1$ otherwise. The value $\det(L_i)\in (\MdZ_2^*)/(\MdZ_2^*)^2\cong C_2\times C_2$ can then be obtained from $\varepsilon$ and the oddity.

Whenever a set of local invariants satisfies the existence conditions given in \cite[Theorem 11]{SPLAG}, a $\MdZ$-lattice with these local invariants exists.

\subsection*{The mass formula}
\label{subsec:massformula}

We assume basic familiarity with the \textsc{Minkowski-Siegel} mass formula, which relates the \emph{mass} of a $\MdZ$-lattice $L$: $$\mass(L)=\sum_{[L_i]\in \Gen(L)}\frac{1}{\#\Aut(L_i)}$$ to the local invariants comprising $\text{sym}(L)$. Recall that all $\MdZ$-lattices $L$ in this article are definite, hence $\Aut(L)$ is always finite.


\begin{definition}\label{defi:masscondition}
We say that a lattice $L$ \emph{fulfils the mass condition} if $\mass(L)\leq \frac{1}{2}$ and $\frac{1}{\mass(L)}\in 2\MdZ$. 
\end{definition}

Clearly, if $L$ is a single-class lattice, then $L$ must fulfil the mass condition. 


We will use a modern formulation of the mass formula, put forth by \textsc{Conway} and \textsc{Sloane} in \cite{massformula}. The following paragraph gives a simplified overview of those parts of the mass formula which are relevant to our computations:

\subsection*{Mass calculation and approximation}
When all local invariants are trivial, i.e. $\det(L)\in\MdZ^*$, the mass of a lattice in dimension $2\leq n \in\MdN$, and of discriminant $D=(-1)^s\det(L)$ where $s:=\left\lceil\frac{n}{2}\right\rceil$, is the \emph{standard mass}: $$\std_n(D) = 2\pi^{-n(n+1)/4}\prod_{j=1}^n\Gamma\left(\frac{j}{2}\right)\left(\prod_{i=1}^{s-1}\zeta(2i)\right)\zeta_D(s).$$ 

Here $\Gamma$ and $\zeta$ denote the usual Gamma and Zeta functions, and $$\zeta_D(s)=\begin{cases}1, & n\text{\ odd}\\ \prod_{2\not=p\in\MdP}\frac{1}{1-\left[\frac{D}{p}\right]p^{-s}}, &\text{otherwise}\end{cases}.$$

where $\left[\frac{D}{p}\right]$ is the \textsc{Legendre} symbol. Finally, $\mass(L)$ is obtained from the standard mass by multiplying with correction factors at all primes $p$ dividing $2\det(L)$. Unlike the standard mass, these depend on the $p$-adic Jordan decomposition of~$L$. Let $\Jordan_p(L)=L_0\bot L_1\bot \dots \bot L_k$ and $s_i := \left\lceil\frac{\rank(L_i)}{2}\right\rceil$, $i=0,\dots,k$. Then

$$\mass(L) = \std_n(D)\cdot\prod_{p| 2\det(L)}\left(m_p(L)\cdot\underbrace{2\prod_{j=2}^s(1-p^{2-2j})}_{=:\  \std_p(L)\iv} \right) \text{\ where\ }$$ 

$$m_p(L)=\left(\prod_{i\in\MdZ:\ \rank(L_i)\not=0}M_p(L_i)\right)\left(\prod_{i<j\in\MdZ}p^{\frac{1}{2}(j-i)\rank(L_i)\rank(L_j)}\right)\hspace{1em}\text{\ (for\ } p\not=2), \text{\ with}$$

$$M_p(L_i)=\frac{1}{2}(1+\varepsilon p^{-{s_i}})\iv\prod_{i=2}^{s_i}(1-p^{2-2i})\iv.$$ 

Here, $\varepsilon=0$ if $\rank(L_i)$ is odd, and otherwise $\varepsilon\in\{-1,1\}$ depends on the species of the orthogonal group $\mathcal{O}_{\rank(L_i)}(p)$ associated with $L_i$, which can be determined from the $p$-adic invariants of $\Gen(L)$.

\section{Single-class lattices}

A fundamental method used in \textsc{Watson}'s classifications is a strategy of descent that transforms a primitive integral quadratic form $f$ into another form $f'$ whose corresponding $\MdZ$-lattice has shorter Jordan decomposition at a given $p\in\MdP$. For the original definition in terms of quadratic forms, cf. \cite[Section~2]{watsontransformations}. We formulate a similar strategy for $\MdZ$-lattices:

\begin{definition}
For $p\in\MdP$, the \emph{Watson mapping}  $\Watp$ is defined by \[\Watp(L) := \Rescale(L\cap pL^\#).\]
\end{definition}

\begin{remark}\label{thm:watsonProcess}The following properties of the $\Watp$ mappings justify the term ``strategy of descent'':
\begin{enumerate}
\item The $\Watp$ mappings are compatible with isometries and extend to well-defined functions on genera of $\MdZ$-lattices with the property $\Gen(\Watp(L)) = \Watp(\Gen(L))$. In particular, the $\Watp$ do not increase the class-number.
\item $\text{len}_p(\Watp(L)) \leq \max\{\text{len}_p(L)-1, 2\}$. Hence, repeatedly applying $\Watp$ decreases the length of a lattice's $p$-adic Jordan decomposition until a $p$-adically square-free lattice $L'$ is reached.
\end{enumerate}
\begin{proof}For details on (2), cf. \cite[(7.4) and (8.4)]{watsontransformations}.\end{proof}
\end{remark}

We will make use of Watson's work one more time by citing the following result:

\begin{theorem}\label{thm:watsonBound}If $L$ is a single-class $\MdZ$-lattice, then $\rank(L)\leq 10$.
\begin{proof}Cf.~\cite{watsonclassnumber}.
\end{proof}
\end{theorem}

\section{Strategy}

Let $L$ be a single-class square-free primitive $\MdZ$-lattice with $\rank(L)=n \geq 3$. We will see in Section~\ref{sec:sqf} that 
the Smith-Minkowski-Siegel mass formula yields an upper bound $\maxprime(n)$ to the set of prime divisors of $\det(L)$. The bound $\maxprime(n)$ depends on $n$ alone (and not on $L$). For the relevant dimensions $3\leq n\leq 10$, its values are given in Table~\ref{table:maximalPrime}. Since $\det(L)$ is thus a-priori bounded and $L$ is square-free, there is only a finite number of possibilities for the genus symbol $\sym(L)$. An algorithm introduced in Section~\ref{constructionoflattices} constructs a $\MdZ$-lattice from a given square-free genus symbol, allowing an enumeration of all single-class square-free $\MdZ$-lattices in any dimension $n\geq 3$.

A second algorithm, based on the $\Watp$ mappings and described in Section~\ref{sec:all}, then yields the complete list of (not necessarily square-free) single-class primitive $\MdZ$-lattices.

\section{Square-free lattices}\label{sec:sqf}

\subsection{Generation of candidate genera}

\begin{lemma}\label{lemma:zeta}
Let $1<s\in\MdN$, $D\in\MdZ$. Then $\zeta_D(s)\geq \frac{\zeta(2s)}{\zeta(s)}$, with $\zeta_D$ defined as in Section~\ref{subsec:massformula}.

\begin{proof}Let $\lambda: \MdN\rarr\{-1,1\}$, $n\mapsto (-1)^{\#\{p\in\MdP:\ p| n\}}$ denote the \textsc{Liouville} function. Then 
\[\zeta_D(s) = \prod_{2\not=p\in \MdP}{\left(1-\left[\frac{D}{p}\right]\frac{1}{p^s}\right)}\iv \geq \prod_{p\in\MdP}\left(1+\frac{1}{p^s}\right)\iv = \sum_{n=1}^\infty\frac{\lambda(n)}{n^s}\]

which converges for any $s>1$. Multiplying by $\zeta(s)$, and writing $\ast$ for {\textsc{Dirichlet}} convolution: \[\zeta(s)\cdot\left(\sum_{n=1}^\infty\frac{\lambda(n)}{n^s}\right) = \sum_{n=1}^\infty\frac{1\ast \lambda(n)}{n^s}=\sum_{n=1}^\infty\frac{\sum_{d|n}\lambda(\frac{n}{d})}{n^s} = \sum_{n=1}^\infty\frac{1}{{(n^2)}^s}=\zeta(2s).\]
\end{proof}

\end{lemma}

\begin{corollary}
\label{corollary:lowerbound}
Whenever $2 < n=2s$ is an even number, there is a lower bound $s(n)$ to the standard mass of a lattice of dimension $n$, independent of its discriminant~$D$:

\[\std_n(D) \geq s(n):= 2\pi^{\frac{-n(n+1)}{4}}\frac{1}{\zeta(s)}\prod_{j=1}^n\Gamma\left(\frac{j}{2}\right)\prod_{j=1}^{s}\zeta(2j).\]
\end{corollary} 

We note that this approach fails for $n=2$, since no similar approximation of $\zeta_D(1)$ is possible. For odd $n$, by definition, $\std_n(D)$ does not depend on $D$ at all.

\begin{lemma}\label{lemma:terminationSCSF}

Let $K$ and $L$ be square-free primitive lattices in dimension $n\geq 3$ whose local invariants differ only at a single prime $p\in \MdP$, where $K$ has trivial invariants and $L$ does not (i.e. $\Jordan_p(K)=K_0$, and $\Jordan_p(L)=L_0\bot L_1$ with $\rank(L_0)>0$ and $\rank(L_1)>0$). Let $s:=\left\lceil\frac{n}{2}\right\rceil$. 

Then ${\mass(L)}\geq a_n(p)\cdot \mass(K)$, where $$a_n(p)= \varepsilon \left(\frac{1}{1+p\iv}\right)^2(\sqrt{p})^{n-1}(1-p^{-2})^{s-1}$$
The factor $\varepsilon$ is $\frac{\zeta(2s)}{{\zeta(s)}^2}$ if $n$ is even, and $1$ otherwise.

\begin{proof}

Let $\Jordan_p(L)=L_0\bot L_1$, $s:=\left\lceil\frac{\rank(L)}{2}\right\rceil$ and $s_k:=\left\lceil\frac{\rank(L_k)}{2}\right\rceil$, $k=0,1$. Then $s \geq 2, s_0, s_1 \geq 1$ and $s_0+s_1\leq s+1$. We have $$\std_p(L)\leq\frac{1}{2}\left(1-p^{-2}\right)^{1-s}$$
and for $k=0, 1$: 
$$M_p(L_k)\geq \frac{1}{2}\left(\frac{1}{1+p\iv}\right)\left(1-p^{2-2s_k}\right)^{1-s_k}$$



Since $(1-p^{2-2s_k})^{1-s_k}>1$ for any $s_k>1$, 

$$\frac{m_p(L)}{\std_p(L)} = M_p(L_0)M_p(L_1)(\sqrt{p})^{\rank(L_0)\rank(L_1)}\frac{1}{\std_p(L)}$$

$$\geq \frac{1}{2}\left(\frac{1}{1+p\iv}\right)^2(\sqrt{p})^{n-1}(1-p^{-2})^{s-1}.$$

Finally, we have $m_q(L)=m_q(K)$ and $\std_q(L)=\std_q(K)$ for all $p\not=q\in\MdP$. Applying Lemma~\ref{lemma:zeta} and the trivial inequality $\zeta_D(s)\leq \zeta(s)$ to the standard masses, we obtain, for $n$ even:
 $$\frac{\mass(L)}{\mass(K)}=\frac{\std_n(\disc(L))}{\std_n(\disc(K))}\cdot\frac{m_p(L)}{\std_p(L)} \geq\frac{\zeta(2s)m_p(L)}{\zeta(s)^2\std_p(L)}.$$
 
For odd $n$, the factor $\zeta_D$ is not present in the standard mass, so in that case $$\frac{\mass(L)}{\mass(K)}\geq \frac{m_p(L)}{\std_p(L)}.$$
\end{proof}
\end{lemma}

\begin{proposition}Let $3\leq n\in\MdN$. Then there is a bound $\maxprime(n)\in\MdN $ such that for any primitive $\MdZ$-lattice $L$ of rank~$n$ which is square-free and fulfils the mass condition (cf. \ref{defi:masscondition}), $\max\{p\in\MdP:\ p\mid \det(L)\}\leq \maxprime(n)$.
\begin{proof}
Let $s(n)$ denote the lower bound to $\std_n(D)$ obtained from Corollary~\ref{corollary:lowerbound}. 

The quantities $m_2(L)$ and $\std_2(L)$ from the mass formula, as defined in \cite[Chapter~15]{SPLAG}, depend only on the $2$-adic genus invariants of~$L$, for which there are finitely many possibilities when $L$ is square-free. Hence $t(n):=\min\{m_2(L)\std_2(L)\iv:\ L\ \mathrm{a\ primitive,\ squarefree}\ \MdZ-\mathrm{lattice\ of\ rank}\ n\}$ is well-defined.

Further, the bound $a_n(p)$ from Lemma~\ref{lemma:terminationSCSF} is increasing in both $n$ and $p$, and $\lim_{p\rightarrow \infty}a_n(p)=\infty$. Let $B(n) := \{2\not=p\in\MdP:\ a_n(p)<1\}$.

Now let $L$ be any $\MdZ$-lattice that is single-class, primitive and square-free. Write $s:=\left\lceil\frac{n}{2}\right\rceil$ and $D:=(-1)^s\det(L)$. Then by the mass formula, and using Lemma~\ref{lemma:terminationSCSF} to compare $\mass(L)$ to the mass of the standard lattice, \begin{eqnarray*}\mass(L)=\std_n(D)\cdot \prod_{p| 2\det(L)}\left(m_p(L)\std_p(L)\iv\right) \geq s(n)t(n)\prod_{2\not=p\mid\det(L)}{a_n(p)}.\end{eqnarray*} 
Hence, if $p\mid\det(L)$ for some $2< p\in \MdP - B(n)$, then \[\mass(L)\geq a_n(p)\cdot s(n)t(n)\prod_{q\in B(n)}{a_n(q)}.\]
So, $\maxprime(n):= \max(\{p\in\MdP: a_n(p)\cdot s(n)t(n)\prod_{q\in B(n)}{a_n(q)} \leq \frac{1}{2}\})$ is the desired bound.

\end{proof}
\end{proposition}

For the relevant dimensions, i.e. $3\leq n\leq 10$ (cf. Theorem~\ref{thm:watsonBound}), the values of $t(n)$, $B(n)$ and $\maxprime(n)$ are given in Table~\ref{table:maximalPrime}. Since the genus of a $\MdZ$-lattice $L$ is determined by local invariants associated to the primes $p|2\det(L)$, for each of which there are finitely many possibilities if $L$ is square-free, we now obtain:

\begin{corollary}
Let $3\leq n\in\MdN$. Then the number of genera of primitive, single-class square-free $\MdZ$-lattices of rank~$n$ is finite.
\end{corollary}

In the remainder of this section, we outline an algorithm to explicitly enumerate the genera of primitive, single-class square-free $\MdZ$-lattices.

\begin{table}[h]
\caption{\label{table:maximalPrime}Upper bound $\maxprime(n)$ for primes dividing $\det(L)$, where $L$ is a square-free $\MdZ$-lattice that fulfils the mass condition} 

\begin{tabular}{lllllllll}
\toprule
$n$ & $3$ & $4$ & $5$ & $6$ & $7$ & $8$ & $9$ & $10$\\
\midrule

$t(n)$ & $8\iv$ & ${24}\iv$ & $8\iv$ & ${72}\iv$ & ${16}\iv$ & ${272}\iv$ & ${32}\iv$ & ${1056}\iv$ \\
$B(n)$ & $\{3\}$ & $\{3\}$ & $\emptyset$ & $\emptyset$ & $\emptyset$ & $\emptyset$ & $\emptyset$ & $\emptyset$\\
$\maxprime(n)$ & $61$ & $467$ & $73$ & $283$ & $139$ & $373$ & $193$ & $421$\\
\bottomrule
\end{tabular}

\end{table}


%
%
%
%
%
%
%

\SetKwFunction{MinimalMass}{MinimalMass}

\begin{proposition}\label{prop:maxprime}
Let $u$ denote a $2$-adic genus symbol, $2\not=q\in\MdP$ and let $v$ denote a list of elements $[p_i,v_i]$, $1\leq i\leq k$, where $p_i\in\MdP$,  $p_i\not\in\{2, q\}$ and $v_i$ is a $p$-adic genus symbol. Denote by $\MinimalMass(u, v, q)$ a lower bound to the mass of any genus of a $\MdZ$-lattice which has the local invariants specified by $u$ and the $v_i$, and which has an unspecified but non-trivial invariant at the prime $q$.

Then the output of the following algorithm, called for an integer $n\geq 3$, contains all primitive single-class square-free genus symbols in dimension~$n$:
\end{proposition}

\begin{algorithm}[H]\label{alg:candidates}
\caption{generating candidate genera for single-class square-free lattices}
\LinesNumbered
\DontPrintSemicolon
\SetKwInOut{Input}{input}
\Input{An integer $n\geq 3$}
\SetKwInOut{Output}{output}
\Output{A list of square-free genus symbols}
\SetKwFunction{NextPrime}{NextPrime}

\SetKwFunction{Mass}{Mass}

\SetKwFunction{IsValidGenusSymbol}{IsValidGenusSymbol}

\SetKw{Append}{append}

\SetKw{Continue}{continue}
\SetKwData{OddSym}{oddsym}
\SetKwData{TwoSym}{twosym}

\Begin{
$\texttt U\longleftarrow$ list of possible $2$-adic invariants for square-free lattices\;
$\texttt V\longleftarrow$ list of possible $p$-adic invariants for square-free lattices, for $2\not=p\in\MdP$\;
$\texttt O\longleftarrow \emptyset$, $\texttt T\longleftarrow \emptyset$\;
\lForEach{$\TwoSym \in \texttt U$}{add $[\TwoSym, \emptyset, 3]$ to $\texttt T$}\;
\While{$\texttt T\not=\emptyset$}{
remove an element $[\texttt{twosym}, \texttt{oddsym}, p]$ from $\texttt T$\;
$q\longleftarrow \NextPrime(p)$\;
\ForEach{$v\in \texttt V$}{
$g\longleftarrow [[2,\TwoSym], \Append(\OddSym, [p, v])]$\;
\lIf{$\IsValidGenusSymbol(g)$ and $\Mass(g)\leq\frac{1}{2}$}{add $g$ to \texttt O}\;
\If{$\MinimalMass(\TwoSym, \Append(\OddSym, [p, v]), q)\leq\frac{1}{2}$}{add $[\TwoSym, \Append(\OddSym, [p, v]), q]$ to \texttt T}
}
\If{$\MinimalMass(\TwoSym, \OddSym, q)\leq\frac{1}{2}$}{add $[\TwoSym, \OddSym, q]$ to \texttt T}
}
\Return $\texttt O$\;

}
\end{algorithm}

\begin{proof}
The two lists generated in steps $2$ and $3$ are finite because we restrict to square-free lattices. The lower bound $\MinimalMass(u,v,q)$ can be evaluated similarly to the value of $\maxprime(n)$ in Proposition~\ref{prop:maxprime}, and is (for fixed $u$ and $v$) increasing in $q$. This ensures that, in steps $12$ and $16$, all primitive square-free genera fulfiling the mass condition are generated by the algorithm.
\end{proof}

\subsection{Construction of square-free lattices}
\label{constructionoflattices}

In this section, we will present an algorithm which generates representative lattices for the candidate genera produced by Algorithm~\ref{alg:candidates}. 


Let $(L, \beta)$ be a definite $\MdZ$-lattice. We view $L$ as embedded into the rational quadratic space $V = L \otimes_\Z \Q$ and denote by $Q_L \colon V \to V, x \mapsto \beta(x,x)$ the corresponding rational quadratic form. We set $\det(Q_L) := \prod a_i\in \MdQ^*/(\MdQ^*)^2$.

The form $Q_L$ is diagonalizable over the rationals, say $Q_L(x) = \sum_{i=1}^n a_i x_i^2$.
For each $p\in\MdP$, we define the local Hasse invariant $c_p(Q_L) = \prod_{i < j} \left(\frac{a_i, a_j}{p}\right)$ where $\left(\frac{a,b}{p}\right) \in \{\pm 1\}$ denotes the usual Hilbert symbol of $a,b$ at $p$.

It is well known that the isometry type of the rational form $Q_L$ is uniquely determined by $n$, $\det(Q_L)=\det(L)\in \MdQ^*/(\MdQ^*)^2$ and the set of all primes $p$ for which $c_p(Q_L) = -1$ (cf. \cite[Remark 66:5]{OMeara}). Further, these invariants can easily be determined from the genus symbol of $L$. Thus, to construct a square-free lattice $L$ with given genus symbol, one can proceed as follows.\vspace{1em}

\begin{algorithm}[H]\label{alg:allSingleClass}
\caption{finding a representative lattice of a square-free genus}
\LinesNumbered
\DontPrintSemicolon
\SetKwInOut{Input}{input}
\Input{A genus symbol $g$ of a square-free $\MdZ$-lattice.}
\SetKwInOut{Output}{output}
\Output{A $\MdZ$-lattice $L$ with genus symbol $g$.}

\Begin{
$P\longleftarrow$ the set of primes at which the enveloping space $L\otimes_\MdZ \MdQ$ has Hasse invariant $-1$\;
$(V, Q)\longleftarrow$ a rational quadratic space of dimension~$\rank(L)$, with $\det(Q)=\det(L)\in\MdQ^*/(\MdQ^*)^2$, and with Hasse invariant $-1$ only at the primes in~$P$\;
$L \longleftarrow $ a lattice $L_0$ in $V$ with $Q(L_0) \subseteq \Z$, and maximal with that property\;
\ForEach{$p\in\MdP$ with  $p|\ 2 \cdot \det(L)$}
{\lIf{$p=2$}{$e_p\longleftarrow 4$} \lElse{$e_p\longleftarrow p$}\;
$L\longleftarrow$ a lattice $L'$ with $e_p L\subseteq L'\subseteq L$ that has $p$-adic genus symbol $g_p$\;
}
\Return $L$\;
}
\end{algorithm}

\begin{proof}
The genus of $L_0$ is uniquely determined by $Q$ and thus by $L$, see for example \cite[Theorem 91:2]{OMeara}.
In particular, $L_0$ contains a sublattice $L$ with genus symbol $g$. Since $L$ is assumed to be integral and square-free, the index $[L_0 : L]$ is at most divisible by $\prod_{p \mid 2\det(L)} e_p$.
\end{proof}

\begin{remark}
The values of $P$, $\det(L)$ and $\rank(L)$ in steps $2$ and $3$ can be read from the genus symbol~$g$.

Further, step~$3$ of the above algorithm can be performed as follows. Let $P'$ be a set of primes containing the prime divisors of $2\cdot \det(L)$.
Then we try diagonal forms $\left \langle a_1,\dots, a_{n-1} , \det(L) \cdot \prod_{i=1}^{n-1} a_i \right\rangle$ where the $a_i$ are products of distinct elements in~$P'$. If the set $P'$ is large enough, this will quickly produce a quadratic form with the correct invariants.

For step~$7$, randomized generation of sublattices between $L$ and $e_pL$ turned out to produce the desired lattice $L'$ quickly enough.
\end{remark}

\section{Completing the classification}\label{sec:all}
The algorithms provided in Section~\ref{sec:sqf} allowed an enumeration of all primitive single-class square-free $\MdZ$-lattices in dimension $3$~--~$10$. In this section, we complete that classification to include all single-class primitive $\MdZ$-lattices in these dimensions, whether square-free or not.

\begin{lemma}\label{lemma:massbound2}
Let $K$ be a definite primitive lattice in dimension $n\geq 3$, $2\not=p\in\MdP$ with $p\not| \det(K)$ and $L\in\Watp\iv(K)$.
\begin{enumerate}
\item Either $\text{len}_p(L)=3$ (more precisely, we have $\Jordan_p(L)=L_0\bot L_2$), or $\Rescale(\Gen(K)) = \Gen(L)$.
\item If $\text{len}_p(L)=3$, then $\mass(L)\geq b_n(p)\cdot\mass(K)$, where $n=\rank(L)$, $s=\left\lceil\frac{n}{2}\right\rceil$ and $$b_n(p)=\varepsilon\left(\frac{1}{1+p\iv}\right)^2p^{n-1}(1-p^{-2})^{s-1}$$ with $\varepsilon=\frac{\zeta(2s)}{2\zeta(s)^2}$ if $\rank(L)$ is even, and $\varepsilon=1$ otherwise.
\end{enumerate}
\end{lemma}

\begin{proof}
The first claim is immediate from the definition of $\Watp$, see \cite[(7.4) and (8.4)]{watsontransformations}. The second claim follows from a calculation similar to Lemma~\ref{lemma:terminationSCSF}.
\end{proof}

\begin{remark}\label{rem:watsoninclusions}Let $L$ be a $\MdZ$-lattice, $p\in\MdP$ and $M\in\Watp\iv(L)$. Then there is some $\alpha\in\MdZ$ such that $M\cap pM^\#$ is the rescaled lattice ${^\alpha}L$.

\[\frac{1}{p}({^\alpha}L)=\frac{1}{p}\Watp(M)=\frac{1}{p}M\cap M^\#\supseteq M \supseteq M\cap pM^\#=\Watp(M)={^\alpha}L.\]

Hence, the pre-images under $\Watp$ correspond to subspaces of $\frac{1}{p}L/L\cong\MdF_p^n$. 
\end{remark}


The following algorithm completes the classification of single-class lattices. 

\begin{algorithm}[H]\label{alg:alllattices}
\caption{Generating single-class lattices}
\LinesNumbered
\DontPrintSemicolon
\SetKwInOut{Input}{input}
\Input{a list {$\mathcal L$} of primitive single-class square-free lattices in dimension $n\geq 2$\;}
\SetKwInOut{Output}{output}
\Output{a list of primitive single-class lattices in dimension $n$}

\Begin{
$\texttt{O}\longleftarrow\mathcal L$, $\texttt{T}\longleftarrow\mathcal L$,  $\texttt{A} \longleftarrow \{\Rescale(\sym(L)):\ L\in \mathcal{L}\}$\;
$L\longleftarrow$ some lattice from $\texttt{T}$, delete $L$ from $\texttt{T}$\;
$\texttt{S}\longleftarrow \{p \in \MdP:\ p| 2\det(L)\text{\ or\ } b_n(p)\cdot\mass(L)\leq\frac{1}{2}\}$, cf. Lemma~\ref{lemma:massbound2}(2)\;
\ForEach{$p\in \texttt{S}$}{

$\texttt U \longleftarrow \{\Rescale(\sym(K)): K \in \Watp\iv(L)\} - \texttt A$\;
$\texttt A \longleftarrow \texttt A \cup \texttt U$\;

\While{$\texttt{U} \not= \emptyset$}{
$L'\longleftarrow$ a random $\MdZ$-lattice with $pL \subseteq L' \subseteq L$\;
$L'\longleftarrow \Rescale(L')$\;
\lIf{$\sym(L')\in \texttt U$}{remove $\sym(L')$ from $\texttt U$} \lElse{go to step 9}\;
\lIf{$L'$ is single-class}{add $L'$ to $\texttt{O}$ and to $\texttt{T}$}
}
}

\lIf{$\texttt{T}=\emptyset$}{\Return $\texttt{O}$} \lElse {go to step 3}

}
\end{algorithm}

\begin{proposition}
Called with a complete list of primitive representatives for single-class \emph{square-free} genera in dimension $n\geq 2$, Algorithm~\ref{alg:alllattices} outputs a  \emph{complete} list of primitive representatives for single-class lattices in dimension $n$.
\begin{proof}
Let $L$ be any primitive single-class $\MdZ$-lattice in dimension~$n$. By Remark~\ref{thm:watsonProcess}(2), $L$ can be reduced to a square-free lattice by repeated application of $\Watp$ mappings. More precisely, there is a list $p_1,\dots,p_k$ of primes and a list $L=L_0, L_1,\dots,L_k$ of $\MdZ$-lattices such that $\Wat_{p_i}(L_{i-1})=L_i$ for all $1\leq i\leq k$ and $L_k$ is square-free. By Remark~\ref{thm:watsonProcess}(1), $L_k$ is also single-class.

$L_k$ is contained in the input $\mathcal L$ by assumption. If $L_i$ ($1\leq i\leq k$) is picked from $\texttt T$ in step~$3$, the set $\texttt S$ computed in step~$4$ will contain $p_i$ by Lemma~\ref{lemma:massbound2}(2) since $L_{i-1}$ is a single-class lattice. A lattice from the (rescaled) isometry class of $L_{i-1}$ will eventually be generated in step~$9$ because of Remark~\ref{rem:watsoninclusions}, and because $\sym(L_{i-1})$ is included in $\texttt U$ in step~$6$. By induction, a lattice isometric to $L$ will be generated by Algorithm~\ref{alg:alllattices}.
\end{proof}
\end{proposition}

An implementation of the above algorithm in \MAGMA produced the complete lists of single-class lattices in dimensions $3-10$ reasonably fast.

\begin{remark}
Since $b_n(p)$ is increasing in $p$, the set $\texttt S$ from step~$4$ is finite and subject to a straightforward computation. The calculation of the full preimage $\Watp\iv(\text{sym}(L))$ in step~$6$ is an easy local process that changes only the invariants for the prime $p$.
\end{remark}

\begin{remark}
In the situation of Lemma~\ref{lemma:massbound2}(2), assuming $D:=\disc(K)$ is known, we have $\disc(L)=p^2D$, and a bound $b_2(p,D)$ similar to $b_n(p)$ can easily be obtained as \[b_2(p,D)=\frac{\zeta_{p^2D}(1)}{2\zeta_{D}(1)}\left(\frac{1}{1+p\iv}\right)^2\cdot p.\]

As a consequence, Algorithm~\ref{alg:alllattices} (step~$4$ in particular) can be modified to apply to dimension~$2$. Thus, while -- to our knowledge -- there is still no way around the Generalized Riemann Hypothesis to classify the $2$-dimensional single-class square-free lattices (cf.~\cite{voight}), this hypothesis is not needed to complete the classification once the single-class square-free lattices are known.
\end{remark}

Algorithm~\ref{alg:alllattices} concludes our classification of single-class $\MdZ$-lattices. In a future publication, we hope to generalize our methods to single-class lattices over arbitrary number fields.

\section*{Acknowledgments}

The present work benefited from the input of Prof.~Gabriele Nebe, RWTH Aachen University. It was made possible by a studentship in DFG Graduiertenkolleg 1632: Experimentelle und konstruktive Algebra.

%

\newpage

\section{Genera of single-class lattices}


\begin{table}[h]\label{table:resultOverview}
\caption{Number of primitive single-class $\MdZ$-lattices}

\begin{tabular}{llllllllll}
\toprule
dimension & $2$ & $3$ & $4$ & $5$ & $6$ & $7$ & $8$ & $9$ & $10$\\
\midrule
\textbf{total} & $1609$ & $794$ & $481$ & $295$ & $186$ & $86$ & $36$ & $4$ & $2$ \\
\hspace{1em} maximal & $769$ & $77$ & $44$ & $16$ & $21$ & $7$ & $6$ & $1$ & $1$\\
\hspace{1em} quadratic-form-maximal & $780$ & $64$ & $20$ & $12$ & $10$ & $5$ & $2$ & $1$ & $1$ \\
maximal prime dividing $\det$. & $409$ & $23$ & $23$ & $11$ & $23$ & $5$ & $5$ & $2$ & $3$\\
maximal determinant & $3{\cdot}5{\cdot} 11{\cdot} 13{\cdot} 19$ & $2^83^37^2$ & $2^43^3{11}^3$ & $2^{12}7^4$ & ${23}^5$ & $2^{18}3^6$ & $3^{14}$ & $2^{24}$ & $3^9$\\
\bottomrule
\end{tabular}

\end{table}

\small




\newenvironment{resulttable}[2]
{ 
\subsubsection{Dimension #1}

#2 lattices:

\begin{longtable}{llllll}
\toprule
Lattice &  & Lattice & & Lattice \\
\midrule
}
{\bottomrule
\end{longtable} 
}


\subsection{How to read the tables}

We give tables of all genus symbols of primitive single-class $\MdZ$-lattices in dimension $\geq 4$. For dimensions $2$ and $3$, and for representative lattices for all the genera in dimensions $3$ -- $10$, see \cite{latdb}. 

The single-class $\MdZ$-lattices have been grouped into families of rescaled partial duals. By a partial dual, we mean $L^{\#,p} := \erz{L, \{v\in L^\#:\ v+L\in S_p\}}$, where $S_p$ denotes the Sylow $p$-subgroup of $L^\#/L$. 
To keep the tables brief, we repeatedly pass to partial duals of $L$, until the lattice with the smallest possible determinant is reached, and print only that lattice in the tables below. In the Catalogue of Lattices available at \cite{latdb}, the list of genera is given in full, and a representative $\MdZ$-lattice for each of these genera is listed.


Printed next to each genus $\Gen(L)$ is a number of the form $\mu^{*N}_{m_1,m_2}$, with $\mu=\mass(L)\iv=\#Aut(L)$, $N$ the number of distinct genera $\Gen(L_1), \dots, \Gen(L_N)$ that can be obtained from $L$ by passing to partial duals of $L$, and the quantities $m_1$ of maximal lattices, and $m_2$ of quadratic-form-maximal lattices (as defined in section~\ref{subsec:preliminaries}) among these. If no subscript is present, both $m_1$ and $m_2$ are $0$.


\subsubsection{Dimension 10}
$2$ lattices:

\begin{longtable}{ll}
\toprule
Lattice &  \\
\midrule
$\text{II}_{10,0}(3^{{-}1})$ & $8360755200_{1,1}^{*2}$ \\

\bottomrule
\end{longtable}

\subsubsection{Dimension 9}
$4$ lattices:

\begin{longtable}{llll}
\toprule
Lattice &  & Lattice &\\
\midrule
$\text{II}_{9,0}(2_{1}^{{+}1})$ & $1393459200_{1,1}^{*2}$
&$\text{II}_{9,0}(8_{1}^{{-}1})$ & $11612160_{}^{*2}$ \\

\bottomrule
\end{longtable}

\begin{resulttable}{8}{36}

$\text{I}_{8,0}$ & $10321920_{1,0}^{*1}$ &$\text{II}_{8,0}$ &
$696729600_{1,1}^{*1}$ &$\text{II}_{8,0}(2_{0}^{{+}4}4^{{+}2})$ &
$18432_{}^{*1}$ \\
$\text{II}_{8,0}(4^{{+}4})$ & $10368_{}^{*1}$ &$\text{II}_{8,0}(4^{{-}4})$ &
$28800_{}^{*1}$ &$\text{II}_{8,0}(2^{{+}4}4^{{+}2})$ & $147456_{}^{*1}$ \\
$\text{II}_{8,0}(2^{{+}4}4^{{-}2})$ & $1327104_{}^{*1}$
&$\text{I}_{8,0}(2^{{-}6}4_{3}^{{-}1})$ & $5806080_{}^{*2}$
&$\text{II}_{8,0}(2^{{+}2})$ & $10321920_{1,0}^{*2}$ \\
$\text{II}_{8,0}(2_{0}^{{-}2})$ & $5806080_{1,0}^{*2}$
&$\text{I}_{8,0}(2^{{-}2})$ & $442368_{1,0}^{*2}$
&$\text{II}_{8,0}(2_{0}^{{-}4})$ & $368640_{}^{*2}$ \\
$\text{II}_{8,0}(2^{{+}4})$ & $2654208_{}^{*1}$ &$\text{II}_{8,0}(4^{{+}2})$ &
$161280_{}^{*2}$ &$\text{II}_{8,0}(4^{{-}2})$ & $622080_{}^{*2}$ \\
$\text{II}_{8,0}(2^{{+}2}4^{{+}2})$ & $36864_{}^{*2}$
&$\text{II}_{8,0}(2^{{+}2}4^{{-}2})$ & $184320_{}^{*2}$
&$\text{II}_{8,0}(2_{0}^{{+}2}4^{{+}2})$ & $17280_{}^{*2}$ \\
$\text{II}_{8,0}(5^{{-}1})$ & $5806080_{1,1}^{*2}$
&$\text{II}_{8,0}(3^{{-}6}9^{{-}1})$ & $1244160_{}^{*2}$
&$\text{II}_{8,0}(3^{{-}2})$ & $1244160_{}^{*2}$ \\
$\text{II}_{8,0}(9^{{+}1})$ & $645120_{}^{*2}$ \\

\end{resulttable}

\begin{resulttable}{7}{86}

$\text{I}_{7,0}$ & $645120_{1,0}^{*1}$ &$\text{I}_{7,0}(2_{1}^{{+}1})$ &
$92160_{1,0}^{*2}$ &$\text{II}_{7,0}(2_{7}^{{-}1})$ & $2903040_{1,1}^{*2}$ \\
$\text{II}_{7,0}(2_{7}^{{-}1}4^{{+}2})$ & $8640_{}^{*2}$
&$\text{II}_{7,0}(2_{0}^{{-}2}8_{7}^{{+}1})$ & $2880_{}^{*2}$
&$\text{II}_{7,0}(2^{{+}2}8_{7}^{{+}1})$ & $9216_{}^{*2}$ \\
$\text{II}_{7,0}(2^{{+}2}8_{7}^{{-}1})$ & $15360_{}^{*2}$
&$\text{I}_{7,0}(4_{3}^{{-}1})$ & $18432_{}^{*2}$
&$\text{II}_{7,0}(4_{7}^{{+}1})$ & $645120_{0,1}^{*2}$ \\
$\text{I}_{7,0}(2^{{-}2})$ & $55296_{1,0}^{*2}$
&$\text{II}_{7,0}(2^{{+}2}16_{7}^{{+}1})$ & $1152_{}^{*2}$
&$\text{II}_{7,0}(2^{{-}4}4_{3}^{{-}1})$ & $92160_{}^{*2}$ \\
$\text{I}_{7,0}(2^{{+}4}4_{3}^{{-}1})$ & $46080_{}^{*2}$
&$\text{I}_{7,0}(2^{{-}4}4_{1}^{{+}1})$ & $92160_{}^{*2}$
&$\text{II}_{7,0}(2_{7}^{{-}3})$ & $92160_{}^{*2}$ \\
$\text{I}_{7,0}(8_{7}^{{+}1})$ & $4608_{}^{*2}$ &$\text{II}_{7,0}(8_{7}^{{-}1})$
& $103680_{}^{*2}$ &$\text{II}_{7,0}(8_{7}^{{+}1})$ & $80640_{}^{*2}$ \\
$\text{I}_{7,0}(2^{{+}4}8^{{+}2})$ & $2880_{}^{*2}$
&$\text{I}_{7,0}(2^{{-}4}8^{{-}2})$ & $8640_{}^{*2}$
&$\text{II}_{7,0}(2^{{+}4}8_{7}^{{+}1})$ & $18432_{}^{*2}$ \\
$\text{II}_{7,0}(2^{{+}4}8_{7}^{{-}1})$ & $55296_{}^{*2}$
&$\text{II}_{7,0}(2_{0}^{{-}4}8_{7}^{{-}1})$ & $1536_{}^{*2}$
&$\text{II}_{7,0}(4^{{+}2}8_{7}^{{-}1})$ & $1440_{}^{*2}$ \\
$\text{II}_{7,0}(4^{{+}2}8_{7}^{{+}1})$ & $864_{}^{*2}$
&$\text{II}_{7,0}(2_{7}^{{+}3}4^{{+}2})$ & $4608_{}^{*1}$
&$\text{II}_{7,0}(16_{7}^{{+}1})$ & $10080_{}^{*2}$ \\
$\text{I}_{7,0}(2^{{+}2}4_{3}^{{-}1})$ & $6144_{}^{*2}$
&$\text{II}_{7,0}(2^{{-}2}4_{3}^{{-}1})$ & $55296_{}^{*2}$
&$\text{II}_{7,0}(2_{0}^{{-}2}4^{{-}2}8_{7}^{{+}1})$ & $576_{}^{*2}$ \\
$\text{II}_{7,0}(2^{{+}2}4^{{+}2}8_{7}^{{+}1})$ & $1536_{}^{*2}$
&$\text{II}_{7,0}(2^{{+}2}4^{{-}2}8_{7}^{{+}1})$ & $4608_{}^{*2}$
&$\text{II}_{7,0}(2^{{+}4}16_{7}^{{+}1})$ & $2304_{}^{*2}$ \\
$\text{I}_{7,0}(3^{{-}1})$ & $46080_{1,0}^{*2}$
&$\text{II}_{7,0}(2_{1}^{{+}1}{\times} 3^{{+}1})$ & $207360_{1,1}^{*4}$
&$\text{II}_{7,0}(2_{7}^{{+}1}{\times} 5^{{-}1})$ & $46080_{1,1}^{*4}$ \\
$\text{II}_{7,0}(4_{5}^{{-}1}{\times} 3^{{-}1})$ & $46080_{0,1}^{*4}$
&$\text{II}_{7,0}(2_{7}^{{-}1}{\times} 9^{{+}1})$ & $7680_{}^{*4}$
&$\text{II}_{7,0}(8_{1}^{{-}1}{\times} 3^{{+}1})$ & $5760_{}^{*4}$ \\

\end{resulttable}

\begin{resulttable}{6}{186}

$\text{I}_{6,0}$ & $46080_{1,0}^{*1}$ &$\text{I}_{6,0}(2_{1}^{{+}1})$ &
$7680_{1,0}^{*2}$ &$\text{II}_{6,0}(4_{7}^{{+}1}8_{7}^{{+}1})$ & $768_{}^{*2}$
\\
$\text{II}_{6,0}(2_{1}^{{+}1}16_{5}^{{-}1})$ & $480_{}^{*2}$
&$\text{II}_{6,0}(2_{7}^{{-}1}16_{7}^{{+}1})$ & $480_{}^{*2}$
&$\text{II}_{6,0}(2_{6}^{{-}2})$ & $46080_{1,1}^{*2}$ \\
$\text{I}_{6,0}(2_{2}^{{+}2})$ & $3072_{1,0}^{*2}$ &$\text{I}_{6,0}(2^{{-}2})$ &
$9216_{1,0}^{*2}$ &$\text{I}_{6,0}(4_{3}^{{-}1})$ & $2304_{}^{*2}$ \\
$\text{I}_{6,0}(2^{{+}2}4_{2}^{{-}2})$ & $1536_{}^{*2}$
&$\text{II}_{6,0}(2^{{-}2}4_{2}^{{+}2})$ & $9216_{}^{*2}$
&$\text{II}_{6,0}(2^{{-}2}4_{2}^{{-}2})$ & $3072_{}^{*2}$ \\
$\text{II}_{6,0}(2_{7}^{{-}3}8_{7}^{{-}1})$ & $384_{}^{*2}$
&$\text{II}_{6,0}(8_{6}^{{+}2})$ & $288_{}^{*2}$
&$\text{II}_{6,0}(8_{6}^{{-}2})$ & $480_{}^{*2}$ \\
$\text{I}_{6,0}(2^{{+}4}4_{3}^{{-}1})$ & $23040_{}^{*1}$
&$\text{II}_{6,0}(2_{6}^{{+}2}4^{{+}2})$ & $2304_{}^{*1}$
&$\text{I}_{6,0}(2_{4}^{{+}4}4_{1}^{{+}1})$ & $1536_{}^{*1}$ \\
$\text{I}_{6,0}(2_{2}^{{+}2}4^{{-}2})$ & $768_{}^{*2}$
&$\text{I}_{6,0}(2^{{+}2}8^{{+}2}16_{7}^{{+}1})$ & $72_{}^{*1}$
&$\text{I}_{6,0}(2^{{-}2}8^{{-}2}16_{7}^{{+}1})$ & $72_{}^{*1}$ \\
$\text{I}_{6,0}(4_{0}^{{+}4}16_{7}^{{+}1})$ & $128_{}^{*1}$
&$\text{I}_{6,0}(4^{{+}4}16_{3}^{{-}1})$ & $2304_{}^{*1}$
&$\text{I}_{6,0}(4^{{+}4}16_{7}^{{+}1})$ & $2304_{}^{*1}$ \\
$\text{I}_{6,0}(8_{1}^{{-}1})$ & $768_{}^{*2}$ &$\text{I}_{6,0}(8_{7}^{{+}1})$ &
$768_{}^{*2}$ &$\text{II}_{6,0}(2_{1}^{{+}1}4_{5}^{{-}1})$ & $7680_{0,1}^{*2}$
\\
$\text{I}_{6,0}(2^{{-}2}4_{4}^{{-}2}16_{7}^{{+}1})$ & $96_{}^{*2}$
&$\text{I}_{6,0}(2^{{-}2}4^{{+}2}16_{3}^{{-}1})$ & $384_{}^{*2}$
&$\text{I}_{6,0}(2^{{+}2}4^{{+}2}16_{7}^{{+}1})$ & $384_{}^{*2}$ \\
$\text{II}_{6,0}(8_{1}^{{-}1}16_{5}^{{-}1})$ & $48_{}^{*2}$
&$\text{I}_{6,0}(2^{{-}4}8_{7}^{{+}1})$ & $7680_{}^{*2}$
&$\text{I}_{6,0}(2^{{-}4}8_{1}^{{-}1})$ & $7680_{}^{*2}$ \\
$\text{I}_{6,0}(8^{{+}4}64_{7}^{{+}1})$ & $72_{}^{*1}$
&$\text{I}_{6,0}(4^{{+}4}32_{7}^{{+}1})$ & $384_{}^{*2}$
&$\text{I}_{6,0}(4^{{+}4}32_{1}^{{-}1})$ & $384_{}^{*2}$ \\
$\text{I}_{6,0}(2^{{+}2}4_{3}^{{-}1})$ & $1536_{}^{*2}$
&$\text{I}_{6,0}(2^{{-}2}4_{1}^{{+}1})$ & $4608_{}^{*2}$
&$\text{II}_{6,0}(2_{7}^{{-}1}8_{7}^{{+}1})$ & $1440_{}^{*2}$ \\
$\text{II}_{6,0}(4_{6}^{{+}2})$ & $4608_{}^{*2}$ &$\text{I}_{6,0}(4_{2}^{{-}2})$
& $512_{}^{*2}$ &$\text{I}_{6,0}(4^{{+}2})$ & $768_{}^{*2}$ \\
$\text{I}_{6,0}(4^{{-}2})$ & $2304_{}^{*2}$ &$\text{II}_{6,0}(4_{6}^{{-}2})$ &
$7680_{}^{*2}$ &$\text{I}_{6,0}(2^{{+}4}32_{7}^{{+}1})$ & $240_{}^{*2}$ \\
$\text{II}_{6,0}(2_{2}^{{+}2}8^{{-}2})$ & $576_{}^{*2}$
&$\text{II}_{6,0}(2_{6}^{{+}2}8^{{+}2})$ & $192_{}^{*2}$
&$\text{II}_{6,0}(2^{{+}2}8_{6}^{{+}2})$ & $256_{}^{*2}$ \\
$\text{II}_{6,0}(2^{{+}2}8_{6}^{{-}2})$ & $768_{}^{*2}$
&$\text{II}_{6,0}(2_{0}^{{+}2}8_{6}^{{+}2})$ & $96_{}^{*2}$
&$\text{II}_{6,0}(2_{7}^{{-}1}4^{{-}2}8_{7}^{{+}1})$ & $288_{}^{*2}$ \\
$\text{I}_{6,0}(2^{{-}4}16_{7}^{{+}1})$ & $1440_{}^{*2}$
&$\text{I}_{6,0}(2^{{-}4}16_{3}^{{-}1})$ & $1440_{}^{*2}$
&$\text{I}_{6,0}(2^{{+}2}8^{{+}2})$ & $192_{}^{*2}$ \\
$\text{I}_{6,0}(2^{{-}2}8^{{-}2})$ & $576_{}^{*2}$ &$\text{I}_{6,0}(3^{{+}1})$ &
$7680_{1,0}^{*2}$ &$\text{I}_{6,0}(3^{{-}1})$ & $4608_{1,0}^{*2}$ \\
$\text{II}_{6,0}(3^{{+}1})$ & $103680_{1,1}^{*2}$ &$\text{I}_{6,0}(5^{{-}1})$ &
$1536_{1,0}^{*2}$ &$\text{II}_{6,0}(7^{{-}1})$ & $10080_{1,1}^{*2}$ \\
$\text{I}_{6,0}(7^{{-}1})$ & $576_{1,0}^{*2}$ &$\text{II}_{6,0}(11^{{+}1})$ &
$3840_{1,1}^{*2}$ &$\text{II}_{6,0}(2^{{+}2}{\times} 3^{{+}1})$ &
$7680_{1,0}^{*4}$ \\
$\text{II}_{6,0}(2_{0}^{{+}2}{\times} 3^{{+}1})$ & $2880_{1,0}^{*4}$
&$\text{II}_{6,0}(2^{{-}2}{\times} 3^{{-}1})$ & $13824_{1,1}^{*4}$
&$\text{I}_{6,0}(2^{{+}2}{\times} 3^{{-}1})$ & $768_{1,0}^{*4}$ \\
$\text{II}_{6,0}(2_{6}^{{-}2}{\times} 5^{{-}1})$ & $1536_{1,1}^{*4}$
&$\text{II}_{6,0}(23^{{-}1})$ & $480_{1,1}^{*2}$ &$\text{I}_{6,0}(3^{{-}3})$ &
$768_{}^{*1}$ \\
$\text{II}_{6,0}(3^{{-}3})$ & $10368_{}^{*1}$
&$\text{II}_{6,0}(3^{{+}1}9^{{+}1})$ & $2304_{}^{*2}$
&$\text{II}_{6,0}(3^{{+}1}9^{{-}1})$ & $2880_{}^{*2}$ \\
$\text{II}_{6,0}(2^{{+}2}{\times} 7^{{-}1})$ & $576_{1,0}^{*4}$
&$\text{II}_{6,0}(2^{{+}2}{\times} 3^{{-}3})$ & $768_{}^{*2}$
&$\text{II}_{6,0}(3^{{+}1}9^{{+}2})$ & $384_{}^{*2}$ \\
$\text{II}_{6,0}(3^{{+}1}9^{{-}2})$ & $192_{}^{*2}$
&$\text{II}_{6,0}(3^{{-}3}9^{{+}1})$ & $1728_{}^{*2}$
&$\text{II}_{6,0}(3^{{-}3}9^{{-}1})$ & $3456_{}^{*2}$ \\
$\text{II}_{6,0}(3^{{+}4}27^{{+}1})$ & $768_{}^{*2}$
&$\text{II}_{6,0}(4^{{-}2}{\times} 3^{{+}1})$ & $864_{}^{*4}$
&$\text{II}_{6,0}(4^{{+}2}{\times} 3^{{+}1})$ & $480_{}^{*4}$ \\
$\text{I}_{6,0}(2^{{-}4}4_{7}^{{+}1}{\times} 5^{{+}1})$ & $768_{}^{*2}$
&$\text{II}_{6,0}(2^{{+}2}4^{{+}2}{\times} 3^{{+}1})$ & $768_{}^{*4}$
&$\text{I}_{6,0}(2^{{+}4}4_{3}^{{-}1}{\times} 3^{{-}1})$ & $2880_{}^{*4}$ \\
$\text{II}_{6,0}(2_{0}^{{-}2}4^{{+}2}{\times} 3^{{+}1})$ & $144_{}^{*2}$
&$\text{II}_{6,0}(3^{{-}1}{\times} 5^{{+}1})$ & $2880_{1,1}^{*4}$
&$\text{II}_{6,0}(3^{{+}1}{\times} 5^{{-}1})$ & $2880_{1,1}^{*4}$ \\

\end{resulttable}

\begin{resulttable}{5}{295}

$\text{I}_{5,0}$ & $3840_{1,0}^{*1}$ &$\text{II}_{5,0}(4^{{+}2}16_{5}^{{-}1})$ &
$96_{}^{*2}$ &$\text{II}_{5,0}(4^{{-}2}16_{5}^{{-}1})$ & $96_{}^{*2}$ \\
$\text{I}_{5,0}(4_{4}^{{-}2}16_{3}^{{-}1})$ & $16_{}^{*2}$
&$\text{II}_{5,0}(4_{2}^{{-}2}16_{3}^{{-}1})$ & $32_{}^{*2}$
&$\text{I}_{5,0}(4^{{+}2}16_{3}^{{-}1})$ & $64_{}^{*2}$ \\
$\text{II}_{5,0}(4_{2}^{{+}2}16_{3}^{{-}1})$ & $96_{}^{*2}$
&$\text{I}_{5,0}(4^{{+}2}16_{7}^{{+}1})$ & $64_{}^{*2}$
&$\text{I}_{5,0}(4^{{-}2}16_{7}^{{+}1})$ & $192_{}^{*2}$ \\
$\text{I}_{5,0}(2^{{-}2}4_{3}^{{-}1}16_{7}^{{+}1})$ & $48_{}^{*2}$
&$\text{I}_{5,0}(2_{1}^{{+}1})$ & $768_{1,0}^{*2}$
&$\text{I}_{5,0}(2^{{-}2}4_{3}^{{-}1}32_{7}^{{+}1})$ & $24_{}^{*2}$ \\
$\text{I}_{5,0}(2^{{-}2}4_{1}^{{+}1}32_{1}^{{-}1})$ & $24_{}^{*2}$
&$\text{II}_{5,0}(4_{6}^{{+}2}32_{7}^{{+}1})$ & $16_{}^{*2}$
&$\text{I}_{5,0}(4^{{+}2}32_{7}^{{+}1})$ & $16_{}^{*2}$ \\
$\text{II}_{5,0}(4_{6}^{{-}2}32_{7}^{{+}1})$ & $48_{}^{*2}$
&$\text{I}_{5,0}(4_{7}^{{+}3}16_{7}^{{+}1})$ & $32_{}^{*1}$
&$\text{I}_{5,0}(8^{{+}2}16_{3}^{{-}1})$ & $24_{}^{*2}$ \\
$\text{I}_{5,0}(8^{{+}2}16_{7}^{{+}1})$ & $24_{}^{*2}$
&$\text{I}_{5,0}(2_{2}^{{+}2})$ & $384_{1,0}^{*2}$
&$\text{I}_{5,0}(4_{3}^{{-}1})$ & $384_{}^{*2}$ \\
$\text{I}_{5,0}(4_{1}^{{+}1})$ & $768_{}^{*2}$ &$\text{II}_{5,0}(4_{5}^{{-}1})$
& $3840_{0,1}^{*2}$ &$\text{I}_{5,0}(2^{{-}2})$ & $2304_{1,1}^{*2}$ \\
$\text{I}_{5,0}(8_{7}^{{+}1})$ & $192_{}^{*2}$ &$\text{I}_{5,0}(8_{7}^{{-}1})$ &
$192_{}^{*2}$ &$\text{I}_{5,0}(8_{1}^{{-}1})$ & $128_{}^{*2}$ \\
$\text{I}_{5,0}(2^{{+}2}4_{1}^{{+}1}8_{1}^{{-}1})$ & $192_{}^{*2}$
&$\text{I}_{5,0}(2^{{-}2}4_{1}^{{+}1}8_{7}^{{-}1})$ & $192_{}^{*2}$
&$\text{I}_{5,0}(2^{{+}2}32_{7}^{{+}1})$ & $24_{}^{*2}$ \\
$\text{I}_{5,0}(4^{{+}2}8_{7}^{{-}1})$ & $192_{}^{*2}$
&$\text{II}_{5,0}(4_{4}^{{-}2}8_{1}^{{+}1})$ & $192_{}^{*2}$
&$\text{II}_{5,0}(4_{4}^{{-}2}8_{1}^{{-}1})$ & $64_{}^{*2}$ \\
$\text{I}_{5,0}(4^{{-}2}8_{7}^{{-}1})$ & $192_{}^{*2}$
&$\text{II}_{5,0}(2_{6}^{{+}2}32_{7}^{{+}1})$ & $24_{}^{*2}$
&$\text{II}_{5,0}(2_{1}^{{+}1}8^{{-}2})$ & $144_{}^{*2}$ \\
$\text{II}_{5,0}(2_{7}^{{-}1}8_{6}^{{+}2})$ & $48_{}^{*2}$
&$\text{II}_{5,0}(2_{2}^{{+}2}16_{3}^{{-}1})$ & $96_{}^{*2}$
&$\text{II}_{5,0}(2_{6}^{{+}2}16_{7}^{{+}1})$ & $96_{}^{*2}$ \\
$\text{II}_{5,0}(2_{0}^{{-}2}16_{5}^{{-}1})$ & $48_{}^{*2}$
&$\text{I}_{5,0}(2_{0}^{{-}2}16_{7}^{{+}1})$ & $32_{}^{*2}$
&$\text{II}_{5,0}(2^{{+}2}16_{5}^{{-}1})$ & $192_{}^{*2}$ \\
$\text{I}_{5,0}(8^{{+}2})$ & $48_{}^{*2}$ &$\text{I}_{5,0}(8^{{-}2})$ &
$144_{}^{*2}$ &$\text{I}_{5,0}(2^{{-}2}16_{7}^{{+}1})$ & $96_{}^{*2}$ \\
$\text{I}_{5,0}(2^{{-}2}16_{3}^{{-}1})$ & $96_{}^{*2}$
&$\text{I}_{5,0}(2_{3}^{{+}3}4_{1}^{{+}1})$ & $192_{}^{*1}$
&$\text{I}_{5,0}(2^{{+}2}8_{1}^{{-}1})$ & $192_{}^{*2}$ \\
$\text{I}_{5,0}(2^{{-}2}8_{7}^{{+}1})$ & $384_{}^{*2}$
&$\text{I}_{5,0}(2^{{-}2}8_{7}^{{-}1})$ & $384_{}^{*2}$
&$\text{I}_{5,0}(2_{2}^{{+}2}8_{7}^{{-}1})$ & $64_{}^{*2}$ \\
$\text{I}_{5,0}(2_{1}^{{+}1}4^{{-}2})$ & $192_{}^{*2}$
&$\text{II}_{5,0}(2^{{-}2}8_{1}^{{-}1})$ & $768_{}^{*2}$
&$\text{II}_{5,0}(2^{{-}2}8_{1}^{{+}1})$ & $2304_{}^{*2}$ \\
$\text{II}_{5,0}(2_{6}^{{+}2}8_{7}^{{+}1})$ & $192_{}^{*2}$
&$\text{I}_{5,0}(2_{2}^{{+}2}4_{1}^{{+}1})$ & $128_{}^{*2}$
&$\text{II}_{5,0}(2_{2}^{{+}2}4_{3}^{{-}1})$ & $384_{}^{*2}$ \\
$\text{I}_{5,0}(4_{2}^{{-}2})$ & $128_{}^{*2}$ &$\text{I}_{5,0}(4^{{-}2})$ &
$384_{}^{*2}$ &$\text{I}_{5,0}(4_{4}^{{-}2})$ & $192_{}^{*2}$ \\
$\text{I}_{5,0}(4_{2}^{{+}2})$ & $384_{}^{*2}$ &$\text{I}_{5,0}(4^{{+}2})$ &
$384_{}^{*2}$ &$\text{II}_{5,0}(16_{5}^{{-}1})$ & $240_{}^{*2}$ \\
$\text{I}_{5,0}(16_{7}^{{+}1})$ & $48_{}^{*2}$
&$\text{I}_{5,0}(2^{{+}2}4_{3}^{{-}1})$ & $768_{}^{*2}$
&$\text{II}_{5,0}(2^{{-}2}4_{1}^{{+}1})$ & $2304_{}^{*2}$ \\
$\text{I}_{5,0}(3^{{+}1})$ & $768_{1,0}^{*2}$ &$\text{I}_{5,0}(3^{{-}1})$ &
$576_{1,0}^{*2}$ &$\text{I}_{5,0}(5^{{-}1})$ & $192_{1,0}^{*2}$ \\
$\text{II}_{5,0}(8_{7}^{{+}1}{\times} 3^{{+}1}9^{{+}1})$ & $16_{}^{*4}$
&$\text{II}_{5,0}(2_{7}^{{+}1}{\times} 3^{{+}1})$ & $1440_{1,1}^{*4}$
&$\text{I}_{5,0}(2_{1}^{{+}1}{\times} 3^{{+}1})$ & $192_{1,0}^{*4}$ \\
$\text{I}_{5,0}(7^{{-}1})$ & $96_{1,0}^{*2}$ &$\text{I}_{5,0}(3^{{-}2})$ &
$192_{}^{*2}$ &$\text{II}_{5,0}(2_{1}^{{+}1}{\times} 5^{{+}1})$ &
$480_{1,1}^{*4}$ \\
$\text{I}_{5,0}(2^{{+}2}{\times} 3^{{-}1})$ & $192_{1,0}^{*4}$
&$\text{II}_{5,0}(4_{7}^{{+}1}{\times} 3^{{+}1})$ & $768_{0,1}^{*4}$
&$\text{I}_{5,0}(4_{7}^{{+}1}{\times} 3^{{-}1})$ & $64_{}^{*4}$ \\
$\text{II}_{5,0}(4_{3}^{{-}1}{\times} 3^{{-}1})$ & $576_{0,1}^{*4}$
&$\text{II}_{5,0}(2_{7}^{{-}1}{\times} 7^{{-}1})$ & $240_{1,1}^{*4}$
&$\text{II}_{5,0}(2_{1}^{{+}1}{\times} 3^{{+}2})$ & $576_{1,1}^{*4}$ \\
$\text{II}_{5,0}(4_{5}^{{-}1}{\times} 5^{{-}1})$ & $192_{0,1}^{*4}$
&$\text{II}_{5,0}(2_{7}^{{-}1}{\times} 11^{{+}1})$ & $96_{1,1}^{*4}$
&$\text{II}_{5,0}(8_{7}^{{-}1}{\times} 3^{{+}1})$ & $144_{}^{*4}$ \\
$\text{II}_{5,0}(8_{7}^{{+}1}{\times} 3^{{+}1})$ & $240_{}^{*4}$
&$\text{I}_{5,0}(3^{{+}1}9^{{-}1})$ & $64_{}^{*2}$
&$\text{II}_{5,0}(4_{7}^{{+}1}{\times} 7^{{-}1})$ & $96_{0,1}^{*4}$ \\
$\text{I}_{5,0}(3^{{-}3}9^{{+}1})$ & $192_{}^{*2}$
&$\text{II}_{5,0}(4_{5}^{{-}1}{\times} 3^{{-}2})$ & $192_{}^{*4}$
&$\text{II}_{5,0}(8_{1}^{{-}1}{\times} 5^{{+}1})$ & $48_{}^{*4}$ \\
$\text{II}_{5,0}(8_{7}^{{+}1}{\times} 3^{{-}3}9^{{+}1})$ & $48_{}^{*2}$
&$\text{II}_{5,0}(16_{7}^{{+}1}{\times} 3^{{+}1})$ & $48_{}^{*4}$
&$\text{II}_{5,0}(2^{{-}2}4_{7}^{{+}1}{\times} 3^{{-}1})$ & $192_{}^{*4}$ \\
$\text{I}_{5,0}(2^{{+}2}4_{7}^{{+}1}{\times} 3^{{-}1})$ & $96_{}^{*4}$
&$\text{I}_{5,0}(2^{{+}2}4_{1}^{{+}1}{\times} 3^{{-}1})$ & $192_{}^{*4}$
&$\text{II}_{5,0}(2_{7}^{{+}1}{\times} 3^{{+}1}9^{{+}1})$ & $96_{}^{*4}$ \\
$\text{II}_{5,0}(2_{7}^{{+}1}{\times} 3^{{+}1}9^{{-}1})$ & $96_{}^{*4}$
&$\text{II}_{5,0}(8_{7}^{{+}1}{\times} 7^{{-}1})$ & $24_{}^{*4}$
&$\text{II}_{5,0}(2_{7}^{{-}1}{\times} 3^{{-}3}9^{{+}1})$ & $288_{}^{*2}$ \\
$\text{II}_{5,0}(2_{7}^{{-}1}{\times} 3^{{-}3}9^{{-}1})$ & $288_{}^{*2}$
&$\text{II}_{5,0}(2_{7}^{{-}1}{\times} 3^{{+}1}9^{{-}2})$ & $32_{}^{*4}$
&$\text{II}_{5,0}(8_{1}^{{-}1}{\times} 3^{{+}2})$ & $96_{}^{*4}$ \\
$\text{II}_{5,0}(2_{7}^{{-}1}4^{{+}2}{\times} 3^{{+}1})$ & $72_{}^{*2}$
&$\text{II}_{5,0}(2_{0}^{{-}2}8_{7}^{{+}1}{\times} 3^{{+}1})$ & $24_{}^{*4}$
&$\text{II}_{5,0}(2^{{+}2}8_{7}^{{+}1}{\times} 3^{{+}1})$ & $192_{}^{*4}$ \\
$\text{II}_{5,0}(2^{{+}2}8_{7}^{{-}1}{\times} 3^{{+}1})$ & $64_{}^{*4}$
&$\text{II}_{5,0}(4_{7}^{{+}1}{\times} 3^{{+}1}9^{{-}1})$ & $64_{}^{*4}$
&$\text{II}_{5,0}(4_{7}^{{+}1}{\times} 3^{{-}3}9^{{-}1})$ & $192_{}^{*4}$ \\
$\text{II}_{5,0}(4^{{-}2}8_{7}^{{+}1}{\times} 3^{{+}1})$ & $24_{}^{*4}$
&$\text{II}_{5,0}(4^{{-}2}8_{7}^{{-}1}{\times} 3^{{+}1})$ & $72_{}^{*4}$
&$\text{II}_{5,0}(2_{1}^{{-}1}{\times} 3^{{-}3}27^{{+}1})$ & $32_{}^{*4}$ \\
$\text{II}_{5,0}(2_{7}^{{+}1}{\times} 3^{{+}1}{\times} 5^{{-}1})$ &
$96_{1,1}^{*8}$ \\

\end{resulttable}

\begin{resulttable}{4}{481}
$\text{I}_{4,0}$ & $384_{1,0}^{*1}$ &$\text{I}_{4,0}(2_{1}^{{+}1})$ &
$96_{1,0}^{*2}$ &$\text{I}_{4,0}(4_{1}^{{+}1}8_{1}^{{+}1})$ & $32_{}^{*2}$ \\
$\text{I}_{4,0}(4_{1}^{{+}1}8_{7}^{{-}1})$ & $16_{}^{*2}$
&$\text{II}_{4,0}(4_{5}^{{-}1}8_{7}^{{-}1})$ & $32_{}^{*2}$
&$\text{I}_{4,0}(4_{1}^{{+}1}8_{1}^{{-}1})$ & $32_{}^{*2}$ \\
$\text{II}_{4,0}(4_{3}^{{-}1}8_{1}^{{+}1})$ & $96_{}^{*2}$
&$\text{I}_{4,0}(32_{7}^{{+}1})$ & $12_{}^{*2}$
&$\text{I}_{4,0}(2^{{+}2}8_{1}^{{-}1})$ & $96_{}^{*2}$ \\
$\text{I}_{4,0}(2^{{-}2}8_{7}^{{-}1})$ & $96_{}^{*2}$
&$\text{II}_{4,0}(2_{1}^{{+}1}16_{3}^{{-}1})$ & $24_{}^{*2}$
&$\text{II}_{4,0}(2_{7}^{{-}1}16_{5}^{{-}1})$ & $24_{}^{*2}$ \\
$\text{I}_{4,0}(2_{2}^{{+}2}8_{1}^{{+}1})$ & $32_{}^{*2}$
&$\text{I}_{4,0}(2_{2}^{{+}2}8_{7}^{{-}1})$ & $32_{}^{*2}$
&$\text{I}_{4,0}(2_{7}^{{-}1}16_{7}^{{+}1})$ & $16_{}^{*2}$ \\
$\text{I}_{4,0}(2_{1}^{{+}1}16_{5}^{{-}1})$ & $16_{}^{*2}$
&$\text{II}_{4,0}(2^{{-}2})$ & $1152_{1,1}^{*1}$ &$\text{I}_{4,0}(2_{2}^{{+}2})$
& $64_{1,0}^{*1}$ \\
$\text{I}_{4,0}(4_{3}^{{-}1})$ & $96_{}^{*2}$ &$\text{I}_{4,0}(4_{1}^{{+}1})$ &
$96_{}^{*2}$ &$\text{II}_{4,0}(8^{{-}2})$ & $72_{}^{*1}$ \\
$\text{I}_{4,0}(4_{3}^{{-}1}16_{3}^{{-}1})$ & $8_{}^{*2}$
&$\text{II}_{4,0}(4_{1}^{{+}1}16_{3}^{{-}1})$ & $24_{}^{*2}$
&$\text{I}_{4,0}(8_{6}^{{+}2})$ & $32_{}^{*1}$ \\
$\text{I}_{4,0}(8_{6}^{{-}2})$ & $32_{}^{*2}$
&$\text{I}_{4,0}(2_{0}^{{-}2}16_{7}^{{+}1})$ & $16_{}^{*2}$
&$\text{I}_{4,0}(2_{0}^{{-}2}16_{3}^{{-}1})$ & $16_{}^{*2}$ \\
$\text{I}_{4,0}(8_{2}^{{-}2}64_{3}^{{-}1})$ & $8_{}^{*2}$
&$\text{I}_{4,0}(8_{6}^{{+}2}64_{7}^{{+}1})$ & $8_{}^{*1}$
&$\text{I}_{4,0}(8_{7}^{{-}1})$ & $32_{}^{*2}$ \\
$\text{I}_{4,0}(8_{1}^{{+}1})$ & $96_{}^{*2}$ &$\text{I}_{4,0}(8_{1}^{{-}1})$ &
$32_{}^{*2}$ &$\text{I}_{4,0}(8_{7}^{{+}1})$ & $96_{}^{*2}$ \\
$\text{I}_{4,0}(2_{1}^{{+}1}4_{1}^{{+}1})$ & $32_{}^{*2}$
&$\text{II}_{4,0}(2_{1}^{{+}1}4_{3}^{{-}1})$ & $96_{0,1}^{*2}$
&$\text{I}_{4,0}(4_{2}^{{-}2}64_{7}^{{+}1})$ & $4_{}^{*2}$ \\
$\text{I}_{4,0}(4^{{-}2}64_{7}^{{+}1})$ & $12_{}^{*2}$
&$\text{I}_{4,0}(4_{7}^{{+}1}32_{7}^{{+}1})$ & $8_{}^{*2}$
&$\text{II}_{4,0}(4_{5}^{{-}1}32_{7}^{{+}1})$ & $24_{}^{*2}$ \\
$\text{II}_{4,0}(4_{3}^{{-}1}32_{1}^{{-}1})$ & $8_{}^{*2}$
&$\text{I}_{4,0}(4_{1}^{{+}1}32_{1}^{{-}1})$ & $8_{}^{*2}$
&$\text{II}_{4,0}(8_{1}^{{-}1}16_{3}^{{-}1})$ & $8_{}^{*2}$ \\
$\text{II}_{4,0}(8_{1}^{{+}1}16_{3}^{{-}1})$ & $24_{}^{*2}$
&$\text{I}_{4,0}(2^{{-}2}32_{1}^{{-}1})$ & $12_{}^{*2}$
&$\text{I}_{4,0}(8_{7}^{{+}1}16_{7}^{{+}1})$ & $16_{}^{*2}$ \\
$\text{I}_{4,0}(8_{1}^{{+}1}16_{5}^{{-}1})$ & $16_{}^{*2}$
&$\text{I}_{4,0}(2_{0}^{{-}2}32_{7}^{{+}1})$ & $8_{}^{*2}$
&$\text{I}_{4,0}(8_{6}^{{+}2}128_{7}^{{+}1})$ & $4_{}^{*2}$ \\
$\text{I}_{4,0}(4_{2}^{{+}2})$ & $64_{}^{*1}$ &$\text{I}_{4,0}(4_{2}^{{-}2})$ &
$64_{}^{*1}$ &$\text{I}_{4,0}(4^{{-}2})$ & $96_{}^{*2}$ \\
$\text{I}_{4,0}(16_{7}^{{+}1})$ & $24_{}^{*2}$ &$\text{I}_{4,0}(16_{3}^{{-}1})$
& $24_{}^{*2}$ &$\text{I}_{4,0}(2_{2}^{{+}2}4_{1}^{{+}1})$ & $32_{}^{*1}$ \\
$\text{I}_{4,0}(2_{1}^{{+}1}8_{7}^{{-}1})$ & $16_{}^{*2}$
&$\text{I}_{4,0}(4_{6}^{{+}2}32_{7}^{{+}1})$ & $16_{}^{*2}$
&$\text{I}_{4,0}(4^{{-}2}32_{7}^{{+}1})$ & $48_{}^{*2}$ \\
$\text{I}_{4,0}(16_{5}^{{-}1}32_{1}^{{-}1})$ & $4_{}^{*2}$
&$\text{I}_{4,0}(4_{4}^{{-}2}32_{1}^{{-}1})$ & $8_{}^{*2}$
&$\text{I}_{4,0}(4^{{+}2}32_{1}^{{-}1})$ & $16_{}^{*2}$ \\
$\text{I}_{4,0}(8_{7}^{{+}1}64_{7}^{{+}1})$ & $4_{}^{*2}$
&$\text{I}_{4,0}(4_{2}^{{+}2}32_{7}^{{-}1})$ & $16_{}^{*2}$
&$\text{I}_{4,0}(4^{{+}2}32_{7}^{{-}1})$ & $16_{}^{*2}$ \\
$\text{I}_{4,0}(4^{{-}2}16_{5}^{{-}1})$ & $96_{}^{*2}$
&$\text{I}_{4,0}(4^{{-}2}16_{1}^{{+}1})$ & $96_{}^{*2}$
&$\text{I}_{4,0}(16_{2}^{{-}2})$ & $8_{}^{*1}$ \\
$\text{I}_{4,0}(16_{6}^{{+}2})$ & $8_{}^{*1}$
&$\text{I}_{4,0}(4_{6}^{{+}2}16_{7}^{{+}1})$ & $16_{}^{*1}$
&$\text{I}_{4,0}(3^{{-}1})$ & $96_{1,0}^{*2}$ \\
$\text{I}_{4,0}(3^{{+}1})$ & $96_{1,0}^{*2}$
&$\text{II}_{4,0}(8_{1}^{{-}1}16_{5}^{{-}1}{\times} 3^{{+}1})$ & $8_{}^{*4}$
&$\text{I}_{4,0}(2^{{-}2}32_{7}^{{+}1}{\times} 3^{{-}1})$ & $4_{}^{*4}$ \\
$\text{II}_{4,0}(4_{5}^{{-}1}32_{1}^{{-}1}{\times} 3^{{+}1})$ & $8_{}^{*4}$
&$\text{II}_{4,0}(5^{{+}1})$ & $240_{1,1}^{*2}$ &$\text{I}_{4,0}(5^{{+}1})$ &
$96_{1,0}^{*2}$ \\
$\text{I}_{4,0}(5^{{-}1})$ & $32_{1,0}^{*2}$
&$\text{I}_{4,0}(2_{1}^{{+}1}{\times} 3^{{+}1})$ & $32_{1,0}^{*4}$
&$\text{I}_{4,0}(4^{{-}2}16_{5}^{{-}1}{\times} 3^{{-}1})$ & $48_{}^{*2}$ \\
$\text{I}_{4,0}(4^{{-}2}16_{1}^{{+}1}{\times} 3^{{-}1})$ & $48_{}^{*2}$
&$\text{I}_{4,0}(4^{{+}2}16_{3}^{{-}1}{\times} 3^{{-}1})$ & $16_{}^{*4}$
&$\text{I}_{4,0}(4_{4}^{{-}2}16_{3}^{{-}1}{\times} 3^{{-}1})$ & $4_{}^{*2}$ \\
$\text{I}_{4,0}(7^{{-}1})$ & $24_{1,0}^{*2}$ &$\text{II}_{4,0}(3^{{+}2})$ &
$288_{1,1}^{*1}$ &$\text{I}_{4,0}(3^{{+}2})$ & $64_{1,0}^{*1}$ \\
$\text{I}_{4,0}(3^{{-}2})$ & $48_{}^{*1}$ &$\text{I}_{4,0}(9^{{-}1})$ &
$32_{}^{*2}$ &$\text{I}_{4,0}(2_{0}^{{-}2}{\times} 3^{{-}1})$ & $32_{2,0}^{*4}$
\\
$\text{II}_{4,0}(2_{6}^{{-}2}{\times} 3^{{+}1})$ & $96_{2,1}^{*4}$
&$\text{II}_{4,0}(2_{2}^{{+}2}{\times} 3^{{-}1})$ & $96_{2,1}^{*4}$
&$\text{I}_{4,0}(4_{7}^{{+}1}{\times} 3^{{-}1})$ & $16_{}^{*4}$ \\
$\text{I}_{4,0}(4_{1}^{{+}1}{\times} 3^{{-}1})$ & $48_{}^{*4}$
&$\text{I}_{4,0}(4^{{-}2}32_{7}^{{+}1}{\times} 3^{{+}1})$ & $16_{}^{*4}$
&$\text{I}_{4,0}(4^{{-}2}32_{1}^{{-}1}{\times} 3^{{+}1})$ & $16_{}^{*4}$ \\
$\text{II}_{4,0}(13^{{+}1})$ & $48_{1,1}^{*2}$
&$\text{I}_{4,0}(2^{{-}2}4_{3}^{{-}1}{\times} 3^{{+}2})$ & $48_{}^{*2}$
&$\text{II}_{4,0}(4^{{+}2}{\times} 3^{{+}2})$ & $16_{}^{*1}$ \\
$\text{II}_{4,0}(4^{{-}2}{\times} 3^{{+}2})$ & $144_{}^{*1}$
&$\text{II}_{4,0}(17^{{-}1})$ & $24_{1,1}^{*2}$
&$\text{II}_{4,0}(2^{{-}2}{\times} 5^{{-}1})$ & $96_{2,1}^{*4}$ \\
$\text{I}_{4,0}(2^{{+}2}{\times} 5^{{-}1})$ & $24_{2,0}^{*4}$
&$\text{II}_{4,0}(4^{{-}2}{\times} 7^{{+}2})$ & $16_{}^{*1}$
&$\text{I}_{4,0}(8_{7}^{{+}1}{\times} 3^{{+}1})$ & $32_{}^{*4}$ \\
$\text{I}_{4,0}(8_{1}^{{-}1}{\times} 3^{{+}1})$ & $32_{}^{*4}$
&$\text{II}_{4,0}(2_{1}^{{+}1}4_{5}^{{-}1}{\times} 3^{{+}1})$ & $32_{0,1}^{*4}$
&$\text{II}_{4,0}(5^{{-}2})$ & $72_{1,1}^{*1}$ \\
$\text{I}_{4,0}(5^{{-}2})$ & $16_{1,0}^{*1}$
&$\text{I}_{4,0}(2^{{+}2}4_{3}^{{-}1}{\times} 3^{{+}2}9^{{-}1})$ & $24_{}^{*2}$
&$\text{II}_{4,0}(4^{{+}2}{\times} 3^{{+}2}9^{{+}1})$ & $8_{}^{*1}$ \\
$\text{II}_{4,0}(4^{{-}2}{\times} 3^{{+}2}9^{{+}1})$ & $72_{}^{*1}$
&$\text{II}_{4,0}(4^{{-}2}{\times} 3^{{+}1}27^{{+}1})$ & $8_{}^{*2}$
&$\text{I}_{4,0}(3^{{+}1}9^{{-}1})$ & $16_{}^{*2}$ \\
$\text{I}_{4,0}(3^{{+}1}9^{{+}1})$ & $32_{}^{*2}$
&$\text{I}_{4,0}(2^{{+}2}{\times} 7^{{-}1})$ & $24_{2,1}^{*4}$
&$\text{I}_{4,0}(2_{0}^{{-}2}{\times} 7^{{-}1})$ & $8_{2,0}^{*2}$ \\
$\text{II}_{4,0}(2^{{+}2}{\times} 3^{{+}2})$ & $64_{1,0}^{*1}$
&$\text{II}_{4,0}(2^{{-}2}{\times} 3^{{-}2})$ & $144_{}^{*1}$
&$\text{II}_{4,0}(2_{0}^{{-}2}{\times} 3^{{+}2})$ & $48_{2,0}^{*2}$ \\
$\text{II}_{4,0}(2^{{-}2}{\times} 9^{{-}1})$ & $96_{}^{*4}$
&$\text{II}_{4,0}(13^{{-}2})$ & $8_{1,1}^{*1}$
&$\text{I}_{4,0}(2_{0}^{{-}2}4_{1}^{{+}1}{\times} 3^{{-}1})$ & $16_{}^{*2}$ \\
$\text{I}_{4,0}(4_{2}^{{-}2}{\times} 3^{{-}1})$ & $16_{}^{*4}$
&$\text{II}_{4,0}(4_{2}^{{-}2}{\times} 3^{{-}1})$ & $32_{}^{*4}$
&$\text{II}_{4,0}(4_{2}^{{+}2}{\times} 3^{{-}1})$ & $96_{}^{*4}$ \\
$\text{I}_{4,0}(4^{{+}2}{\times} 3^{{+}1})$ & $32_{}^{*4}$
&$\text{II}_{4,0}(4_{6}^{{+}2}{\times} 3^{{+}1})$ & $96_{}^{*4}$
&$\text{I}_{4,0}(2^{{-}2}4_{1}^{{+}1}{\times} 3^{{+}1})$ & $48_{}^{*2}$ \\
$\text{I}_{4,0}(2^{{+}2}4_{7}^{{+}1}{\times} 3^{{-}1})$ & $48_{}^{*2}$
&$\text{II}_{4,0}(2_{7}^{{-}1}8_{7}^{{-}1}{\times} 3^{{+}1})$ & $12_{}^{*4}$
&$\text{II}_{4,0}(7^{{+}2})$ & $32_{1,1}^{*1}$ \\
$\text{I}_{4,0}(4^{{-}2}{\times} 3^{{+}1}9^{{+}1})$ & $32_{}^{*4}$
&$\text{II}_{4,0}(4_{6}^{{+}2}{\times} 3^{{+}1}9^{{-}1})$ & $16_{}^{*4}$
&$\text{I}_{4,0}(2^{{+}2}4_{3}^{{-}1}{\times} 3^{{-}1}9^{{-}1})$ & $16_{}^{*2}$
\\
$\text{II}_{4,0}(2_{7}^{{+}1}8_{7}^{{+}1}{\times} 3^{{+}1}9^{{+}1})$ &
$4_{}^{*4}$ &$\text{I}_{4,0}(2^{{+}2}4_{7}^{{+}1}{\times} 3^{{-}1}9^{{+}1})$ &
$8_{}^{*2}$ &$\text{II}_{4,0}(8_{6}^{{+}2}{\times} 3^{{+}1})$ & $24_{}^{*4}$ \\
$\text{II}_{4,0}(8_{6}^{{-}2}{\times} 3^{{+}1})$ & $8_{}^{*4}$
&$\text{I}_{4,0}(2^{{+}2}16_{3}^{{-}1}{\times} 3^{{-}1})$ & $12_{}^{*4}$
&$\text{I}_{4,0}(2^{{+}2}16_{7}^{{+}1}{\times} 3^{{-}1})$ & $12_{}^{*4}$ \\
$\text{II}_{4,0}(2^{{+}2}{\times} 3^{{+}2}9^{{-}1})$ & $32_{}^{*2}$
&$\text{II}_{4,0}(2^{{-}2}{\times} 9^{{-}2})$ & $16_{}^{*1}$
&$\text{II}_{4,0}(2^{{-}2}{\times} 9^{{+}2})$ & $64_{}^{*1}$ \\
$\text{II}_{4,0}(2^{{-}2}{\times} 3^{{-}2}9^{{-}1})$ & $144_{}^{*2}$
&$\text{II}_{4,0}(2_{0}^{{+}2}{\times} 3^{{+}2}9^{{+}1})$ & $24_{}^{*2}$
&$\text{I}_{4,0}(8^{{-}2}{\times} 3^{{-}1}9^{{-}1})$ & $8_{}^{*4}$ \\
$\text{I}_{4,0}(2^{{+}2}16_{7}^{{+}1}{\times} 3^{{-}1}9^{{-}1})$ & $4_{}^{*4}$
&$\text{I}_{4,0}(2^{{+}2}16_{3}^{{-}1}{\times} 3^{{-}1}9^{{-}1})$ & $4_{}^{*4}$
&$\text{II}_{4,0}(4^{{-}2}{\times} 5^{{+}1})$ & $24_{}^{*4}$ \\
$\text{I}_{4,0}(2^{{+}2}4_{1}^{{+}1}{\times} 5^{{-}1})$ & $24_{}^{*4}$
&$\text{II}_{4,0}(3^{{+}2}9^{{-}1})$ & $144_{}^{*1}$
&$\text{II}_{4,0}(3^{{+}2}9^{{+}1})$ & $144_{}^{*1}$ \\
$\text{II}_{4,0}(3^{{+}1}27^{{+}1})$ & $16_{}^{*2}$
&$\text{I}_{4,0}(3^{{+}2}9^{{-}1})$ & $32_{}^{*1}$
&$\text{I}_{4,0}(3^{{+}2}9^{{+}1})$ & $32_{}^{*1}$ \\
$\text{I}_{4,0}(3^{{-}2}9^{{+}1})$ & $48_{}^{*2}$
&$\text{II}_{4,0}(2^{{-}2}{\times} 5^{{+}2}25^{{+}1})$ & $32_{}^{*2}$
&$\text{II}_{4,0}(3^{{+}1}9^{{-}1}27^{{+}1})$ & $24_{}^{*2}$ \\
$\text{II}_{4,0}(3^{{+}1}9^{{+}1}27^{{+}1})$ & $24_{}^{*1}$
&$\text{I}_{4,0}(3^{{+}1}9^{{+}1}27^{{-}1})$ & $16_{}^{*1}$
&$\text{II}_{4,0}(4_{7}^{{+}1}8_{7}^{{+}1}{\times} 3^{{+}1})$ & $32_{}^{*4}$ \\
$\text{I}_{4,0}(2^{{-}2}8_{7}^{{+}1}{\times} 3^{{-}1})$ & $32_{}^{*4}$
&$\text{I}_{4,0}(2^{{+}2}8_{1}^{{-}1}{\times} 3^{{-}1})$ & $32_{}^{*4}$
&$\text{II}_{4,0}(2_{1}^{{+}1}16_{5}^{{-}1}{\times} 3^{{+}1})$ & $8_{}^{*4}$ \\
$\text{II}_{4,0}(2_{7}^{{-}1}16_{7}^{{+}1}{\times} 3^{{+}1})$ & $8_{}^{*4}$
&$\text{II}_{4,0}(2^{{+}2}{\times} 5^{{-}2})$ & $16_{1,0}^{*1}$
&$\text{II}_{4,0}(2^{{-}2}{\times} 5^{{+}2})$ & $64_{}^{*1}$ \\
$\text{II}_{4,0}(2^{{-}2}{\times} 11^{{-}2})$ & $16_{}^{*1}$
&$\text{I}_{4,0}(2^{{-}2}4_{5}^{{-}1}{\times} 3^{{-}1}9^{{-}1}27^{{-}1})$ &
$4_{}^{*2}$ &$\text{II}_{4,0}(4^{{-}2}{\times} 3^{{+}1}9^{{+}1}27^{{+}1})$ &
$12_{}^{*1}$ \\
$\text{I}_{4,0}(2^{{+}2}{\times} 3^{{-}1}9^{{+}1})$ & $16_{}^{*4}$
&$\text{I}_{4,0}(2^{{+}2}{\times} 3^{{-}1}9^{{-}1})$ & $32_{}^{*4}$
&$\text{I}_{4,0}(2^{{-}2}4_{3}^{{-}1}{\times} 7^{{-}1})$ & $12_{}^{*2}$ \\
$\text{II}_{4,0}(4_{6}^{{+}2}{\times} 7^{{-}1})$ & $8_{}^{*4}$
&$\text{I}_{4,0}(4^{{-}2}{\times} 7^{{-}1})$ & $24_{}^{*4}$
&$\text{II}_{4,0}(7^{{+}2}49^{{-}1})$ & $8_{}^{*1}$ \\
$\text{II}_{4,0}(5^{{-}2}25^{{+}1})$ & $24_{}^{*2}$
&$\text{II}_{4,0}(2^{{-}2}{\times} 3^{{-}1}9^{{-}1}27^{{+}1})$ & $48_{}^{*2}$
&$\text{II}_{4,0}(2_{0}^{{+}2}{\times} 3^{{+}1}9^{{+}1}27^{{+}1})$ & $4_{}^{*2}$
\\
$\text{II}_{4,0}(5^{{+}1}25^{{+}1})$ & $16_{}^{*2}$
&$\text{II}_{4,0}(5^{{+}1}25^{{-}1})$ & $24_{}^{*2}$
&$\text{II}_{4,0}(3^{{+}1}{\times} 23^{{-}1})$ & $8_{1,1}^{*4}$ \\
$\text{II}_{4,0}(3^{{-}2}9^{{+}1}{\times} 5^{{+}1})$ & $24_{}^{*2}$
&$\text{II}_{4,0}(3^{{+}2}9^{{-}1}{\times} 5^{{-}1})$ & $48_{}^{*4}$
&$\text{I}_{4,0}(3^{{+}1}{\times} 5^{{-}1})$ & $16_{1,0}^{*4}$ \\
$\text{II}_{4,0}(3^{{-}1}{\times} 7^{{+}1})$ & $48_{1,1}^{*4}$
&$\text{II}_{4,0}(3^{{+}1}{\times} 7^{{-}1})$ & $48_{1,1}^{*4}$
&$\text{II}_{4,0}(3^{{+}1}9^{{+}1}27^{{+}1}{\times} 5^{{-}1})$ & $8_{}^{*4}$ \\
$\text{II}_{4,0}(3^{{+}1}{\times} 11^{{+}1})$ & $16_{1,1}^{*4}$
&$\text{I}_{4,0}(2^{{+}2}4_{7}^{{+}1}{\times} 3^{{-}1}{\times} 5^{{+}1})$ &
$8_{}^{*4}$ &$\text{II}_{4,0}(4_{6}^{{-}2}{\times} 3^{{+}1}{\times} 5^{{-}1})$ &
$16_{}^{*8}$ \\
$\text{II}_{4,0}(3^{{+}2}{\times} 5^{{-}1})$ & $48_{1,1}^{*4}$
&$\text{II}_{4,0}(9^{{+}1}{\times} 5^{{+}1})$ & $16_{}^{*4}$
&$\text{II}_{4,0}(3^{{+}1}9^{{+}1}{\times} 7^{{-}1})$ & $16_{}^{*4}$ \\
$\text{II}_{4,0}(3^{{+}1}9^{{-}1}{\times} 7^{{-}1})$ & $8_{}^{*4}$
&$\text{II}_{4,0}(4^{{-}2}{\times} 3^{{+}1}{\times} 11^{{+}1})$ & $8_{}^{*4}$
&$\text{I}_{4,0}(2^{{+}2}{\times} 3^{{-}1}{\times} 5^{{+}1})$ & $16_{2,1}^{*8}$
\\

\end{resulttable}

\vspace{1cm}

\normalsize 

%
%
%
%
%
%
%
%
%


\newpage

\bibliographystyle{amsplain}
\bibliography{literatur/literatur}

\vfill

\end{document}